\documentclass[11pt]{article}
\usepackage[top=1in, bottom=1in, left=1in, right=1in]{geometry}

\usepackage[linktocpage,colorlinks,linkcolor=blue,anchorcolor=blue,citecolor=blue,urlcolor=blue,pagebackref]{hyperref}
\usepackage{euscript,mdframed}
\usepackage{microtype,todonotes,relsize}
\usepackage{amsmath,amsthm, amssymb}
\usepackage{algpseudocode}

\usepackage{accents}
\usepackage{cleveref}
\usepackage{comment,url,graphicx,relsize}
\usepackage{amssymb,amsfonts,amsmath,amsthm,amscd,dsfont,mathrsfs,mathtools,nicefrac,bm}
\usepackage{float,psfrag,epsfig,color,url}
\usepackage{epstopdf,bbm,mathtools,enumitem}
\usepackage{subfigure}
\usepackage[ruled,vlined]{algorithm2e}
\usepackage{tablefootnote}
\graphicspath{{Plots/}}
\usepackage{diagbox}
\usepackage{tikz}
\usetikzlibrary{calc,patterns,angles,quotes}
\usepackage{makecell}
\usepackage{multirow}
\usepackage{booktabs}
\usepackage{tcolorbox, mathtools}
\usepackage[dvipsnames]{xcolor}

\newcommand{\bbV}{\mathbb{V}}
\newcommand{\RR}{\mathbb{R}}

\providecommand{\argmin}{\mathop\mathrm{arg min}}

\ifdefined\nonewproofenvironments\else
\ifdefined\ispres\else
\renewenvironment{proof}{\noindent\textbf{Proof.}\hspace*{.3em}}{\qed\\}
\newenvironment{proof-sketch}{\noindent\textbf{Proof Sketch}
  \hspace*{0.em}}{\qed\bigskip\\}
\newenvironment{proof-idea}{\noindent\textbf{Proof Idea}
  \hspace*{0.em}}{\qed\bigskip\\}
\newenvironment{proof-of-lemma}[1][{}]{\noindent\textbf{Proof of Lemma {#1}.}
  \hspace*{0.em}}{\qed\\}
\newenvironment{proof-of-corollary}[1][{}]{\noindent\textbf{Proof of Corollary {#1}.}
  \hspace*{0.em}}{\qed\\}
\newenvironment{proof-of-theorem}[1][{}]{\noindent\textbf{Proof of Theorem {#1}.}
  \hspace*{0.em}}{\qed\\}
\newenvironment{proof-attempt}{\noindent\textbf{Proof Attempt}
  \hspace*{0.em}}{\qed\bigskip\\}

\fi

\newtheorem{theo}{Theorem}[section]
\newtheorem{theorem}[theo]{Theorem}
\newtheorem{lemma}[theo]{Lemma}

\newtheorem{proposition}[theo]{Proposition}
\newtheorem{assumption}[theo]{Assumption}

\newtheorem{definition}[theo]{Definition}

\usetikzlibrary{calc}

\allowdisplaybreaks
\usepackage[authoryear,round]{natbib}
\renewcommand*{\backref}[1]{\ifx#1\relax \else Page #1 \fi}
\renewcommand*{\backrefalt}[4]{%
  \ifcase #1 \footnotesize{(Not cited.)}%
  \or        \footnotesize{(Cited on page~#2.)}%
  \else      \footnotesize{(Cited on pages~#2.)}%
  \fi
}

\newcommand{\Lf}{\ensuremath{L_{f,1}}}
\newcommand{\Lff}{\ensuremath{L_{f,2}}}
\newcommand{\Lfff}{\ensuremath{L_{f,3}}}
\newcommand{\Lg}{\ensuremath{L_{g,1}}}
\newcommand{\Lgg}{\ensuremath{L_{g,2}}}
\newcommand{\Lggg}{\ensuremath{L_{g,3}}}
\newcommand{\Lgggg}{\ensuremath{L_{g,4}}}

\newcommand*{\colorboxed}{}
\def\colorboxed#1#{%
  \colorboxedAux{#1}%
}
\newcommand*{\colorboxedAux}[3]{%
  \begingroup
    \colorlet{cb@saved}{.}%
    \color#1{#2}%
    \boxed{%
      \color{cb@saved}%
      #3%
    }%
  \endgroup
}
\numberwithin{equation}{section}
\usepackage{xcolor} 
\usepackage{marginnote}

\newcommand{\todol}[2][]{{%
 \let\marginpar\marginnote
 \reversemarginpar
 \renewcommand{\baselinestretch}{0.8}%
 \todo[color=yellow]{#2}}}

\title{{ On Second-Order Methods for  Bilevel Optimization}}
\author{
{Jiawen Bi} \thanks{Department of Industrial and System Engineering, University of Minnesota.  \texttt{bi000050@umn.edu}}
\and
{Jiaxiang Li} \thanks{Department of Electrical and Computer Engineering, University of Minnesota.  \texttt{jasonli.optmal@gmail.com}}
\and 
{Mingyi Hong} \thanks{Department of Electrical and Computer Engineering, University of Minnesota.  \texttt{mhong@umn.edu}}
\and 
{Shuzhong Zhang} \thanks{Department of Industrial and System Engineering, University of Minnesota. \texttt{zhangs@umn.edu}}
}
\date{}

\begin{document}
\maketitle

\begin{abstract}
Bilevel optimization is an indispensable 
modeling tool for modern machine learning and engineering design. 
However, the theory and practice for finding second order stationary points in the context of bilevel optimization still remain largely unsettled. 
Even for bilevel optimization with strongly convex lower-level problem, the hyperfunction it induces is in 
general nonconvex. 
Although the Cubic Regularized Newton methods (CRN) famously achieve the optimal $\mathcal{O}(\varepsilon^{-1.5})$ SOSP (second-order stationary point) rate in single-level optimization, it is unclear 
how to control the accuracy of the hypergradient and hyper-Hessian computations in the context of 
applying the second-order methods to bilevel problems in order for the 
overall process to be efficient. In this paper, we set out to answer this question. In particular, we first formulate a {\it double loop}\/ CRN baseline that achieves the optimal outer rate but requires repeated lower level solves. Next, we propose a {\it single loop}\/ cubic regularized Newton algorithm that combines one lower-level gradient step with one Newton step for the hypergradient, and prove an overall deterministic $\mathcal{O}(\varepsilon^{-1.5})$ total oracle complexity, which is optimal. 
In addition, we illustrate that some intuitively simple modifications of our method may fail to hold up the convergence result. 
To the best of our knowledge, this is the first deterministic single loop method for unconstrained NCSC (non-convex upper-level and strongly convex lower-level) bilevel optimization setting that achieves the $\mathcal{O}(\varepsilon^{-1.5})$ optimal convergence rate for finding an $\varepsilon$-SOSP of the hyperfunction.
\end{abstract}


\section{Introduction}\label{sec:1}

Bilevel optimization provides a powerful framework for modeling hierarchical decision making, where an outer objective depends implicitly on the solution to an inner optimization problem. It has emerged as a central paradigm in modern machine learning, underpinning diverse applications such as hyperparameter optimization \citep{maclaurin2015gradient,franceschi2018bilevel}, meta-learning \citep{finn2017model, franceschi2018bilevel}, neural architecture search \citep{yin2022bm,tu2024efficient}, reinforcement learning \citep{zeng2024demonstrations,hong2022twotimescaleframeworkbileveloptimization}, and, more recently, machine unlearning \citep{reisizadeh2025blur,fan2024challenging} and language model alignment \citep{chakraborty2023parl,shen2023penaltybasedbilevelgradientdescent}, as well as Operations Research problems such as pricing \citep{colson2007overview,colson2005bilevel}, transportation design \citep{brotcorne2001bilevel} and game theory \citep{Silvrio2022ABO, labbe2016bilevel}.  
Formally, in this paper we focus on the nonconvex-strongly-convex (NCSC) unconstrained bilevel problem of the form:
\begin{align}\label{eq:bilevel_problem}
    \min_{x\in\mathbb{R}^m}  \Phi(x) := f(x,y^*(x))
    \quad\text{s.t.}\quad
        y^*(x) = \arg\min_{y\in\mathbb{R}^n} g(x,y),
\end{align}
where $f$ and $g$ represent the upper and lower level objectives respectively. The lower level function $g(x, y)$ is smooth and strongly convex with respect to $y$ and the resulting hyperfunction $\Phi(x)$ is in general nonconvex, and $x, y$ are known as the upper and lower level variables respectively in the literature. As a matter of fact, 
\eqref{eq:bilevel_problem}
is arguably the most well studied case in the bilevel optimization literature. 
Solving \eqref{eq:bilevel_problem} with gradient based methods hinges on computing the hypergradient $\nabla \Phi(x)$. Since $y^*(x)$ is defined implicitly, the hypergradient is typically derived via the chain rule, which requires differentiating the inner problem's optimality condition $\nabla_y g(x, y^*(x)) = 0$. This yields the gradient of the hyperfunction, which we call hypergradient, as follows: 
\begin{equation}\label{eq:hypergradient}
    \nabla \Phi(x) = \nabla_x f(x, y^*(x)) - \nabla_{xy}^2 g(x, y^*(x)) \left[ \nabla_{yy}^2 g(x, y^*(x)) \right]^{-1} \nabla_y f(x, y^*(x)).
\end{equation}
In practice, computing \eqref{eq:hypergradient}
exactly is 
impractical because $y^*(x)$ cannot be exactly computed. 
A standard technique to get around this difficulty is known as {\it 
Approximate Implicit Differentiation}\/ (AID) \citep{ghadimi2018approximationmethodsbilevelprogramming}, which approximates the hypergradient by first approximating $y^*(x)$ with an iterative solver and then approximating the Hessian inverse-vector product using methods such as the Neumann series approach \citep{hong2022twotimescaleframeworkbileveloptimization} or the conjugate gradient method \citep{ji2021bilevel}.

Assuming access to an approximate hypergradient, the goal of optimization algorithms is to find points that satisfy certain stationarity conditions. Suppose that $\Phi(x)$ is twice differentiable. Given an accuracy level $\varepsilon$, 
the standard notions of the first and second order $\varepsilon$-stationarity conditions are defined 
as follows: 
\begin{definition}[$\varepsilon$-FOSP]\label{def:fosp}
    An $\varepsilon$-first order stationary point (FOSP) is a point $x^*$ such that
\begin{align}
    \|\nabla \Phi(x^*)\|\le \varepsilon. \nonumber
\end{align}
\end{definition}

\begin{definition}[$\varepsilon$-SOSP]\label{def:sosp}
 An $\varepsilon$-second order stationary point (SOSP) is a point $x$ such that
\begin{align}\label{eq:so_stationary}
    \|\nabla \Phi(x)\|\le\varepsilon,\ \lambda_{\min}(\nabla^2 \Phi(x))\geq -\sqrt{\varepsilon}
\end{align}
where $\lambda_{\min}(\nabla^2 \Phi(x))$ denotes the minimal eigenvalue of $\nabla^2 \Phi(x)$.
\end{definition}
In bilevel optimization, even when $g(x,\cdot)$ is strongly convex, the induced hyperfunction $\Phi(x)=f(x,y^*(x))$ might be nonconvex. Therefore, first order bilevel methods only guarantees to converge to saddle points of $\Phi$. The goal of this paper is to obtain deterministic convergence to an $\varepsilon$-SOSP of $\Phi$, matching the best-known rate for second order single level nonconvex optimization.

In unconstrained nonconvex single level optimization, an SOSP solution 
is both desirable and achievable. 
There are essentially two approaches to reaching an SOSP. The first is known as Perturbed Gradient Descent \citep{ge2015escaping,jin2017escape}, which adds 
a Gaussian perturbation when the iterates approach an approximate stationary point. It can be shown that the iterates converge with high probability to an $\varepsilon$-SOSP with an iteration complexity of $\mathcal{O}(\varepsilon^{-2}(\log \frac{d}{\varepsilon})^4)$. 
The second approach is the Cubic Regularized Newton method\/ (CRN) ~\citep{nesterov2006cubic}, which adds a cubic regularizer term to the second order Taylor expansion and updates with the minimizer of such regularized subproblem. CRN converges deterministically to an $\varepsilon$-SOSP with iteration rate $\mathcal{O}(\varepsilon^{-1.5})$. 

Regarding the bilevel optimization problem \eqref{eq:bilevel_problem}, there is a gap between existing results and the best convergence we could expect. \cite{huang2025efficiently} introduced the Perturbed AID and iNEON algorithms by adapting perturbed gradient and NEON techniques and established convergence to $\varepsilon$-SOSPs at a rate of $\mathcal{O}(\varepsilon^{-2}\log(\varepsilon^{-1}))$ for NCSC (non-convex upper-level and strongly convex lower-level) bilevel and NCSC minimax problems. 
Later, \cite{yang2023accelerating,xian2025escaping} proposed accelerated algorithms based on this framework and obtained faster convergence rate of $\mathcal O(\varepsilon^{-1.75}\log(\varepsilon^{-1}))$. 
However, these prior results are first order methods with only high probability assurance instead of deterministic convergence; moreover, the convergence rates are suboptimal in reference to the lower bound $\mathcal O(\varepsilon^{-1.5})$ for SOSP. Naturally, one would in principle prefer deterministic convergence if at all possible. 
Along that line, very recently \cite{yang2026second} proposed penalty function second order methods, which find an SOSP with an $\mathcal O(\varepsilon^{-1.5})$ outer iteration complexity; however, the inner problem still requires to solve a subproblem, therefore, the overall complexity is $\mathcal O(\varepsilon^{-1.5}\log(\varepsilon^{-1}))$.  Since all prior convergence rates fail to reach the exact lower bound, and that first order methods can only guarantee convergence with high probability, it is natural to ask whether one can achieve deterministic convergence to an SOSP 
with a rate matching the lower bound. 
In this paper, we set out to answer the above question affirmatively. In particular, we propose two methods: DLCRN (double-loop cubic regularized Newton) and SLCRN (single-loop cubic regularized Newton). Both of them converge to an SOSP deterministically, and SLCRN achieves the optimal convergence rate to an SOSP of the hyperfunction.

\subsection{Our Contributions}
Our primary contributions are summarized as follows.
\begin{itemize}[leftmargin=1.6em]
\item \textbf{Optimal–rate second order convergence.}  
      We show that our proposed methods (DLCRN and SLCRN) deterministically find an $\varepsilon$–second order stationary point of the hyperfunction in
          {$\mathcal{O}\bigl(\varepsilon^{-1.5}\bigr)$}
      outer loop oracle complexity (Proposition \ref{prop:conv_dl} and \ref{thm:main}),  $\mathcal{O}\bigl(\varepsilon^{-1.5}\log(\varepsilon^{-1})\bigr)$ and $\mathcal{O}\bigl(\varepsilon^{-1.5}\bigr)$ inner loop oracle complexity, respectively (Theorem \ref{thm:dl_ll} and \ref{thm:main}). 
      The convergence rate of SLCRN matches the optimal complexity for second order nonconvex optimization in terms of $\varepsilon$. This improves on previous bilevel methods which require random perturbations and only offer high probability guarantees \citep{huang2025efficiently,yang2023accelerating}. It also improves the fastest convergence rate of first order methods by a factor of $\mathcal{O}(\varepsilon^{-0.25}\log (\varepsilon^{-1}))$.

\item \textbf{Simple Algorithm Design and Efficient Tuning.} While we analyze the DLCRN algorithm as an extension of the Cubic Regularized Newton Method on bilevel problems, its algorithm structure demands an inner loop subproblem solving procedure, and its stopping criterion is difficult to verify in certain practical applications. SLCRN circumvents this bottleneck by replacing the inner loop by simply running one gradient descent step and one Newton step. Moreover, the design of SLCRN's lower level update is critical to its convergence and stability. We provide high-level insights and construct a counterexample to explain why two intuitive simplifications 
might lose the optimal convergence rate, confirming that our lower level update structure is necessary to maintain the optimal convergence rate. Moreover, SLCRN requires tuning with only two parameters, which saves considerable time compared to tuning multiple variables such as stepsizes and the parameters regarding the inner loop stopping criteria.

\end{itemize}

\subsection{Related Work}\label{sec:literature}

The results in the literature that are most related to ours include \cite{huang2025efficiently,yang2023accelerating}, which also aim to finding a second order stationary point under the bilevel framework. In particular, \cite{huang2025efficiently} adheres to  the principle of perturbation in escaping saddle points in general nonconvex optimization problem, and achieves a high-probability convergence to an SOSP with a convergence rate of $\mathcal{O}(\varepsilon^{-2})$ in NCSC bilevel 
problems.
In the context of the NCSC minimax problem, which is a special case of our bilevel setting, several works have explored Cubic Regularized Newton methods. \cite{chen2021cubic} and \cite{luo2022finding} both propose double loop algorithms that, while deterministic, are complex to implement and analyze. Furthermore, their approaches for computing the hyper-Hessian are tailored specifically to the minimax structure and do not readily extend to the general bilevel case. In contrast, our proposed algorithm SLCRN is single loop, deterministic, and applicable to the broader class of NCSC bilevel problems. 
Below we provide more information for the benefit of interested readers. 

\subsubsection*{Minimax and Bilevel Optimization}
Unconstrained BLO has recently seen remarkable advancements through gradient-based approaches (AID and ITD) (see e.g.~\cite{ghadimi2018approximationmethodsbilevelprogramming,ji2021bilevel,hong2022twotimescaleframeworkbileveloptimization,chen2022singletimescalemethodstochasticbilevel}) and function value based approaches (see e.g.~\cite{lu2024firstorderpenaltymethodsbilevel,shen2025penalty,kwon2023penalty,jiang2025discretization,liu2022bome}). 
Gradient based approaches reformulate the problem around a single variable hyperfunction and apply gradient descent variants. A central challenge in this paradigm is to compute the hypergradient $\nabla \Phi(x)$, as it requires access to the exact optimal solution at the lower level $y^*(x)$, which is rarely available in practice. Two prominent approximation techniques have been developed: Approximate Implicit
Differentiation (AID), for example in \cite{ghadimi2018approximationmethodsbilevelprogramming,ji2021bilevel,jiang2024barrier}, and Iterative Differentiation (ITD), for example in \cite{franceschi2017forwardreversegradientbasedhyperparameter}, \cite{franceschi2018bilevel}. The popular AID approach utilizes the implicit function theorem, which typically requires the solution of a linear system derived from the optimality conditions of the inner problem. Under the standard assumption that the inner problem is strongly convex, AID-based algorithms have been shown to achieve strong convergence guarantees  (cf.~e.g.~\cite{ghadimi2018approximationmethodsbilevelprogramming,hong2022twotimescaleframeworkbileveloptimization,ji2021bilevel,chen2022singletimescalemethodstochasticbilevel}). Motivated by the application scenarios where the strong convexity is too restrictive, recent efforts have focused on relaxing this assumption, extending bilevel analysis to inner problems that satisfy the PL condition  (see e.g.~\cite{chen2025set,chen2024finding}), or are locally Morse (see e.g.~\cite{bolte2025bilevel}), or are generally nonconvex (see e.g.~\cite{jiang2025correspondence}).

On the other hand, an alternative paradigm avoids direct hypergradient estimation by reformulating the bilevel program as a single level problem. This is typically achieved by incorporating the inner problem's objective as a penalty term in the outer objective. Such methods are attractive because they bypass the complication of needing to differentiate through the inner problem's solution map, therefore extending the applicability to settings where the inner objective is merely convex (see e.g.~\cite{lu2024firstorderpenaltymethodsbilevel}) or merely satisfies the PL condition (see e.g.~\cite{xiao2024unlocking}). 
However, a critical theoretical limitation of many penalty-based approaches is the potential mismatch between the stationary points of the surrogate objective and those of the true hyperfunction $\Phi(x)$. As noted in \cite{chen2024finding}, convergence to a stationary point of the penalized function with first order methods does not guarantee even the first order stationarity for the original bilevel problem, leaving the quality (or fidelity) of the solutions so-obtained up in the air for discussion.

\subsubsection*{Escaping Saddle Points and High Order Methods for Single level Optimization}
Finding an $\varepsilon$-SOSP in general nonconvex optimization is a well-studied problem. One of the most prominent approaches is the Cubic Regularized Newton (CRN) method, introduced by \cite{nesterov2006cubic}. Foundational work established its fast convergence for convex problems and its ability to find SOSPs in the nonconvex case \citep{nesterov2008accelerating}. Subsequent research has extended the original framework to include adaptive regularization parameters \citep{cartis2011adaptive2,he2025history}, stochastic finite-sum objectives \citep{tripuraneni2018stochastic}, momentum-based acceleration \citep{chayti2024improving}, variational inequalities \citep{huang2022cubic}.

A key challenge in CRN is how to efficiently solve the cubic subproblem at each iteration. Various strategies have been proposed to ensure solving the subproblems practically efficient. The original work suggested collapsing the problem into a one-dimensional root-finding task \citep{nesterov2006cubic}. More recent approaches have focused on avoiding explicit Hessian factorization, using iterative methods like gradient descent on the subproblem \citep{carmon2019gradient} or solving it within a low-dimensional Krylov subspace to achieve dimension-free iteration bounds \citep{jiang2024krylov}. Another line of research has focused on relaxing the need for exact gradient and Hessian information, developing conditions under which inexact Hessians still permit convergence \citep{cartis2011adaptive, wang2018note}.

Besides the CRN, some other forms of methods utilizing high-order information have been proposed as well; for example, high order regularized methods \citep{huang2025approximation}; trust region based methods \citep{jiang2023beyond,zhang2022drsom}; homogeneous second-order methods \citep{he2025homogeneous,zhang2025homogeneous}. These methods also provably escape saddle points. An alternative line of research focuses on first-order methods but using stochastic perturbations to escape saddle points, including the Perturbed Gradient Descent Method which finds an $\varepsilon$-SOSP with high probability \citep{ge2015escaping, jin2017escape, jin2021nonconvex, fang2019sharp}. While these algorithms are computationally cheaper per iteration, they generally have slower convergence rates as compared to the CRN, typically by a factor of $\mathcal{O}(\varepsilon^{-0.5})$ in the worst-case scenario. Moreover, the first-order based methods are guaranteed to converge to an SOSP only probabilistically, rather than deterministically as in the case of higher-order methods.

\section{Finding Second Order Stationary Points: the Cubic Regularized Newton Approach}\label{sec:2}

In this section, we establish the theoretical foundations required in finding an SOSP in the bilevel optimization framework. We begin by reviewing the standard Cubic Regularized Newton method and highlight how the exact hypergradient and hyper-Hessian could be derived via the Implicit Function Theorem if we were to apply CRN to bilevel optimization. Since calculating these exact derivatives is computationally intractable due to numerical error in real-world scenarios, we then resort to the Approximate Implicit Differentiation (AID) approach via 
inexact lower level solutions. Finally, we establish the Lipschitz continuity properties of these approximations for our subsequent algorithms and convergence analysis in the next section.

\subsection{Cubic Regularized Newton for Bilevel Optimization}\label{sec:crn_for_bo}
The Cubic Regularized Newton (CRN) method, introduced by \cite{nesterov2006cubic}, augments the classical Newton step with a scalar cubic term that guarantees global convergence to an SOSP in nonconvex settings. 
For a single level function $F(x)$, at each iterate $x_k$, let $\bm{G}_k=\nabla F(x_k)$ and $\bm{H}_k=\nabla^{2}F(x_k)$, and 
CRN computes a step $s_k\in\mathbb{R}^m$ as the global minimizer of the cubic subproblem and performs the update as follows
\begin{align}
    s_k &=
    \argmin_{s\in\mathbb{R}^m}
    \Bigl\{
        \bm{G}_k^{\top}s
        +\dfrac12\,s^{\top}\bm{H}_k s
        +\dfrac{M}{6}\,\|s\|^{3}
    \Bigr\}, \nonumber\\
    x_{k+1} &= x_k + s_k, \nonumber
\end{align}
where $M>0$ is a regularization parameter satisfying 
$M\ge L_{\nabla^2 F}$ (the Hessian Lipschitz constant) in single level optimization problems.
\cite{nesterov2006cubic} showed that the CRN method converges to an $\varepsilon$-SOSP deterministically with the iteration rate of $\mathcal O(\varepsilon^{-1.5})$.

To apply the classical CRN to bilevel optimization, we need to construct the gradient and Hessian of the hyperfunction. Notice that for bilevel problem where the lower level is strongly convex, we can compute the hyperfunction gradient through the Implicit Function Theorem and the chain rule, as mentioned in Lemma 2.1 in \cite{ghadimi2018approximationmethodsbilevelprogramming}.
Given $g$'s strong convexity, we take the derivative of $g$ with respect to $x$ and obtain
\begin{align*}
    \nabla_y g(x,y^*(x))=0.
\end{align*}
Taking the gradient with respect to $x$ on both sides simultaneously, we obtain
\begin{align*}
    \nabla^2_{xy}g(x,y^*(x))+\nabla_x y^*(x)\nabla^2_{yy}g(x,y^*(x))=0.
\end{align*}
When the lower level function $g(x, y)$ is strongly convex, $\nabla^2_{yy}g(x,y)$ is invertible, which implies
\begin{align}
    \nabla_x y^*(x)=-\nabla^2_{xy}g(x,y^*(x))(\nabla^2_{yy}g(x,y^*(x)))^{-1}. \nonumber
\end{align}
By denoting $H$ as $\nabla^2_{yy} g(x,y^*(x))$, and defining the hyperfunction $\Phi(x)=f(x,y^*(x))$, we can express the gradient of $\Phi(x)$ with respect to $x$ as
\begin{align*}
    \nabla_x \Phi(x) &= \nabla_x f(x,y^*(x))+\nabla_xy^*(x)\nabla_y f(x,y^*(x))\\
    &=\nabla_x f(x,y^*(x))-\nabla^2_{xy}g(x,y^*(x))H^{-1}\nabla_y f(x,y^*(x)).
\end{align*}
Similarly, we can obtain \begin{align}
        \nabla^2_{xx}\Phi(x) = \ & \nabla^2_{xx}f(x, y^*(x)) - \nabla^2_{xy}g(x, y^*(x))H^{-1}\nabla^2_{xy}f(x, y^*(x))^\top\notag\\
    &- \nabla^2_{xy}g(x, y^*(x))H^{-1}\left(\nabla^2_{xy}f(x, y^*(x)) - \nabla^2_{xy}g(x, y^*(x))H^{-1}\nabla^2_{yy}f(x, y^*(x)) \right)^\top \notag\\
    &- \nabla^2_{xy}g(x, y^*(x))\left(H^{-1}\nabla^3_{yxy} g(x, y^*(x))H^{-1}\right)\nabla_yf(x, y^*(x)) \notag\\
    &- \left(\nabla^3_{xxy}g(x, y^*(x)) -\nabla^2_{xy}g(x,y^*(x))H^{-1}\nabla^3_{yxy}g(x, y^*(x))\right)H^{-1}\nabla_yf(x, y^*(x)),\nonumber
    \end{align}
where the third order derivative tensor of a function $r(x, y)$ is given by \begin{align*}\nabla^3_{xxy} r(x,y) = \left[\dfrac{\partial^3 r}{\partial x_i\partial x_j\partial y_k}\right]_{m\times m\times n},
    \nabla^3_{yxy}r(x,y) = \left[\dfrac{\partial^3f}{\partial y_i\partial x_j\partial y_k}\right]_{n\times m\times n}, 
    \nabla^3_{yyy}r(x, y) = \left[\dfrac{\partial^3f}{\partial y_i\partial y_j\partial y_k}\right]_{n\times n\times n}.\end{align*}
Note that the multiplications between tensor $A\in\mathbb{R}^{n_1\times n_2\times n_3}$ and column vector $b\in\mathbb{R}^{n_3\times 1}$, and tensor $A\in\mathbb{R}^{n_1\times n_2\times n_3}$ and a matrix $B\in\mathbb{R}^{n_3\times n_4}$ are:  
\(    
\left(A b\right)_{ij} = \sum_{k=1}^{n_3} A_{ijk}b_k,\, AB = [A_{[i, :, :]}B]_{i\in [n_1]}.
\)

Notice that computing the gradient and the Hessian requires the exact value of the lower level solution $y^*(x)$, which is almost impossible to compute in practice. 
Instead, we could expect to obtain some numerically approximate solution $\tilde y$ of the lower level problem within a certain accuracy, and then use that approximative solution $\tilde y$ of $y^*(x)$ to construct some gradient and Hessian proxies for the upper level updates.
To obtain gradient and Hessian information per outer iteration, we refer to the Approximate Implicit Differentiation (AID) scheme.

 For the AID scheme, we first perform a few inner updates to produce an inexact lower level point $\tilde y$, then evaluate the closed form implicit function formulas at $(x,\tilde y)$ to approximate the true hypergradient $\nabla_x\Phi$ and the hyper-Hessian $\nabla^{2}_{xx}\Phi$. The approximation of the hypergradient $\nabla_x\Phi$ is well-studied (e.g.~in \cite{ghadimi2017second, ji2021bilevel}), which is formulated as follows:  
 \begin{align}
    \widehat{\nabla}_x\Phi(x, \tilde y) = \nabla_xf(x, \tilde y) - \nabla_{xy}^2g(x, \tilde y)\left(\nabla_{yy}^2g(x, \tilde y)\right)^{-1}\nabla_yf(x, \tilde y). \nonumber
\end{align}

Similarly we could derive the approximation hyper-Hessian function at the inexact lower level point $\tilde y$ as follows
\begin{align}
        \widehat{\nabla}^2_{xx}\Phi(x, \tilde y) := & \nabla^2_{xx}f(x, \tilde y) - \nabla^2_{xy}g(x, \tilde y)H^{-1}\nabla^2_{xy}f(x, \tilde y)^\top \notag\\
    &- \nabla^2_{xy}g(x, \tilde y)H^{-1}\left(\nabla^2_{xy}f(x, \tilde y) - \nabla^2_{xy}g(x, \tilde y)H^{-1}\nabla^2_{yy}f(x, \tilde y) \right)^\top \notag\\
    &- \nabla^2_{xy}g(x, \tilde y)\left(H^{-1}\nabla^3_{yxy} g(x, \tilde y)H^{-1}\right)\nabla_yf(x, \tilde y) \notag\\
    &- \left(\nabla^3_{xxy}g(x, \tilde y) -\nabla^2_{xy}g(x,\tilde y)H^{-1}\nabla^3_{yxy}g(x, \tilde y)\right)H^{-1}\nabla_yf(x, \tilde y),\label{eq:def_approx_hess}
    \end{align}
where $H = \nabla_{yy}^2g(x, \tilde y)$. Notice that in the computation of exact and approximate hyper-Hessian, we might have to use tensor information.

With the exact and approximate hypergradient and hyper-Hessian derived above, we need some mild assumptions to guarantee the convergence of inexact CRN on the hyperfunction. In the following Section \ref{sec:lip} we specify the main assumptions and prove the Lipschitz continuity of the approximate hypergradient and hyper-Hessian.

\subsection{Lipschitz Continuity of Approximate Hypergradient and Hyper-Hessian}
\label{sec:lip}

To facilitate further discussion, we have the following somewhat standard assumptions regarding the upper level objective function $f$ and the lower level objective function $g$, mostly on their Lipschitz smoothness property.  
As noted in 
\cite{huang2025efficiently},
such conditions are essential for identifying SOSP's in bilevel optimization. 
First of all, let us assume sufficient continuous differentiability of upper and lower level objective functions, as well as the strong convexity of $g(x, \cdot)$.
\begin{assumption}\label{assumption:general1}
    $f(x, y)$ is twice continuously differentiable, and $g(x, y)$ is three times continuously differentiable.
\end{assumption}

\begin{assumption}\label{assumption:nonlinear2} $g(x,y)$ is $\mu_g$-strongly convex in $y$ for any $x\in\RR^m$.
\end{assumption}
To locate an SOSP \eqref{eq:so_stationary} via AID approach, we further assume the Lipschitz continuity of upper level function $f(x, y)$ and third order derivative of lower level problem $g(x, y)$, which are common and indispensable assumptions to prove the Lipschitz continuity of $\widehat \nabla_{xx}^2\Phi(x, y)$ (defined in \eqref{eq:def_approx_hess}) with respect to $x$ and $y$ (see \cite{huang2025efficiently}). 

\begin{assumption}\label{assumption:general2}The function, gradient, and Hessian of upper level function $f$ are Lipschitz continuous, and the function, gradient, Hessian and third order derivatives of lower level function $g$ are Lipschitz continuous. 
In other words, the following inequalities hold for every $x,\overline{x}\in\RR^m,y,\overline{y}\in\RR^n$: 
\begin{enumerate}
    \item $\|\nabla_x f(x,y)\|\le \Lf$; $\|\nabla_y f(x,y)\|\le \Lf$;
    \item $\|\nabla^2_{xx}f(x, y)\|\le \Lff$;\quad$\|\nabla^2_{xy}f(x, y)\|\le \Lff$; \quad$\|\nabla^2_{yy}f(x, y)\|\le \Lff$;
    \item $\|\nabla^2_{xx} f(x,y)-\nabla^2_{xx} f(\overline{x},\overline{y})\|\le \Lfff\|(x,y)-(\overline{x},\overline{y})\|$;\\$\|\nabla^2_{xy} f(x,y)-\nabla^2_{xy} f(\overline{x},\overline{y})\|\le \Lfff\|(x,y)-(\overline{x},\overline{y})\|$;
    \item $\|\nabla_x g(x,y)\|\le \Lg$; $\|\nabla_y g(x,y)\|\le \Lg$;
    \item $\|\nabla^2_{yy} g(x,y)\|\le\Lgg$; \quad$\|\nabla^2_{xy} g(x,y)\|\le \Lgg$;
    \item $\|\nabla^3_{xxy} g(x,y)\|\le\Lggg$; \quad$\|\nabla^3_{xyy} g(x,y)\|\le\Lggg$;\quad$\|\nabla^3_{yyy} g(x,y)\|\le\Lggg$;
    \item $\|\nabla^3_{xxy} g(x,y) - \nabla^3_{xxy} g(\overline{x},\overline{y})\|\le\Lgggg\|(x,y) - (\overline{x}, \overline{y})\|$;\\ $\|\nabla^3_{yxy} g(x,y) - \nabla^3_{yxy} g(\overline{x},\overline{y})\|\le\Lgggg\|(x,y) - (\overline{x}, \overline{y})\|$;\\$\|\nabla^3_{yyy} g(x,y) - \nabla^3_{yyy} g(\overline{x},\overline{y})\|\le\Lgggg\|(x,y) - (\overline{x}, \overline{y})\|$.
\end{enumerate}
\end{assumption}

With the above assumptions and the approximate hyper-Hessian $\widehat{\nabla}_{xx}^2\Phi(x, y)$, below we establish that the approximate hyper-Hessian function 
$\widehat{\nabla}^{2}_{xx}\Phi(x,y)$ is jointly Lipschitz in $(x,y)$ in Lemma \ref{lem:lip_xy}. Notice that the Lipschitz continuity of $\nabla_{xx}^2\Phi(x)$ is already given in \cite{huang2025efficiently}. However, because we are exploiting inexact lower level solution $\tilde y$, we need an explicit bound of the accuracy of hyper-Hessian. Additionally, we also need to establish the Lipschitz continuity of $\widehat{\nabla}^2_{xx}\Phi(x,y)$ with respect to $y$. 
The proof the lemma is relegated to Appendix \ref{sec:pf_lip_xy}

\begin{lemma}\label{lem:lip_xy}
    Under Assumption \ref{assumption:general1}, \ref{assumption:nonlinear2} and \cref{assumption:general2}, $\widehat{\nabla}^2_{xx}\Phi(x, y)$ is Lipschitz continuous.
    In other words, there exists a positive constant $L_{\widehat{\nabla}^2_{xx}\Phi}$ such that for all $x_1, x_2, y_1, y_2$ \begin{align}
        \|\widehat{\nabla}^2_{xx}\Phi(x_1, y_1) - \widehat{\nabla}^2_{xx}\Phi(x_2, y_2)\| \le L_{\widehat{\nabla}^2_{xx}\Phi} \|(x_1, y_1) - (x_2, y_2)\|. \nonumber
    \end{align}
\end{lemma}

Next, in Lemma \ref{lemma:GH_lip} we quantify how
lower level inaccuracy propagates to the
hypergradient and hyper-Hessian constructed by AID. This is a straightforward consequence of Assumptions \ref{assumption:general1} and \ref{assumption:general2}. In Proposition \ref{prop:lPhi}, we show that the hyper-Hessian $\nabla^2_{xx}\Phi(x)$ is Lipschitz continuous with respect to $x$. The proofs for Lemma \ref{lemma:GH_lip} and Proposition \ref{prop:lPhi} can be found in Appendix \ref{sec:pf_GHPhi} and Appendix \ref{sec:pf_prop_lPhi} respectively.

\begin{lemma}\label{lemma:GH_lip}
Under Assumptions \ref{assumption:general1}, \ref{assumption:nonlinear2}  \ref{assumption:general2}, the approximate gradient $\widehat{\nabla}_x\Phi(x,y)$ and Hessian $\widehat{\nabla}^2_{xx}\Phi(x,y)$ of the hyperfunction are Lipschitz continuous with respect to $y$. In particular, suppose $\tilde y$ is an approximate solution of the lower level problem, then there exist constants $L_G, L_H > 0$ such that for all $x, \tilde y$, $\|\widehat{\nabla}_x\Phi(x,\tilde y)-\nabla_x\Phi(x)\|\le L_G\|\tilde{y} - y^*(x)\|$, and $\|\widehat{\nabla}^2_{xx}\Phi(x,\tilde y)-\nabla^2_{xx}\Phi(x)\|\le L_H\|\tilde{y} - y^*(x)\|$ for all $x$.

\end{lemma}

\begin{proposition}\label{prop:lPhi}
    Under Assumptions \ref{assumption:general1}, \ref{assumption:nonlinear2} and \ref{assumption:general2}, the hyper-Hessian $\nabla_{xx}^2\Phi(x)$ is Lipschitz continuous with constant $L_{\nabla^2\Phi}$.
\end{proposition}

The above analysis identifies a few key quantitative ingredients that are needed to adapt CRN to the bilevel setting. In particular, the approximate hypergradient and hyper-Hessian inherit errors due to the lower level inaccuracies, while the exact hyper-Hessian indeed remains Lipschitz continuous along the upper level trajectory. 
The next question is: How should one stipulate update rules for the lower level iterates so as to ensure the approximation errors are small enough for the entire inexact cubic regularized Newton approach to perform well? We shall address this question in the next section through a double loop construction.

\section{A Double Loop Cubic Regularized Newton Method}\label{sec:3}

Under the assumptions and preliminary results from the previous section, we apply the CRN method to bilevel optimization and propose a Double Loop Cubic Regularized Method (DLCRN, Algorithm \ref{alg:dl_crn_blo}) for the bilevel optimization problem in this section. This approach represents an adaptation of the inexact CRN method \citep{wang2018note} to the bilevel setting, and it is included as a diagnostic baseline. It shows what accuracy in the lower-level variable is sufficient for inexact CRN on the hyperfunction: the hypergradient error must scale quadratically with the CRN step, while the hyper-Hessian error only needs to scale linearly. This exposes the bottleneck that a single loop method must remove, which we will discuss later in Section~\ref{sec:gd_only}.

Applying CRN methods to bilevel optimization algorithms with AID-based estimates is directly related to the literature on inexact CRN methods for single level optimization problem, since for every outer iteration we can only obtain an approximate solution $\tilde y$ of the lower level objective, and hence the hypergradient and hyper-Hessian are inexact. Existing theoretical analysis for inexact CRN establish the optimal $\mathcal{O}(\varepsilon^{-1.5})$ convergence rate to an SOSP under the condition that the gradient and Hessian estimation errors satisfy:
\begin{align}
    \|\bm{G}_k - \nabla \Phi(x_k)\| \le \mathcal{O}(\|s_k\|^2), \qquad
    \|\bm{H}_k - \nabla^2\Phi(x_k)\| \le \mathcal{O}(\|s_k\|), \label{eq:ine_crn_con}
\end{align}
where $s_k$ is the CRN step at iteration $k$ (see Theorem 1 in \cite{wang2018note}).

Based on the theoretical framework in \cite{wang2018note}, we design a Double Loop algorithm (Algorithm \ref{alg:dl_crn_blo}).
This algorithm provides a theoretically transparent baseline for second order bilevel optimization. By explicitly solving the lower level problem to a prescribed accuracy at each outer iteration, the resulting approximate hypergradient and hyper-Hessian strictly satisfy the inexactness conditions in \eqref{eq:ine_crn_con}. This mechanism ensures that the algorithm naturally inherits the optimal outer loop complexity  $\mathcal{O}(\varepsilon^{-1.5})$ of the inexact CRN, establishing a principled benchmark that clarifies the exact relationship between inner loop precision and outer loop convergence. However, the inexact conditions \eqref{eq:ine_crn_con} are merely theoretical requirements, since they depend on $\|y_{k,t} - y^*(x_k)\|$ and are not directly verifiable. In practice, one can replace the stopping criteria by $\|\nabla_y g(x_k, y_{k, t})\| \le \mu_g\min\left\{\dfrac{a}{L_{\nabla\Phi}}\|s_k\|^2, \dfrac{b}{L_{\nabla^2\Phi}}\|s_k\|\right\}$.

\begin{algorithm}[h]\label{alg:dl_crn_blo}
\caption{A Double Loop Cubic Regularized Newton algorithm for BLO problem (DLCRN)}
\begin{algorithmic}
    \State {\bf Input:} iteration number $K$, stepsize $\beta$, CRN parameter $M$, constants $a, b \ge 0$
    \State {\bf Initialization:} initial point $x_1, y_1, s_1$
    \State \For {$k =1 \to K$}{
    
    \State $y_{k, 0} = y_k, t = 0$;
    \State \While{$\|\nabla_yg(x_k, y_{k, t})\| > \mu_g\min\left\{\dfrac{a}{L_{\nabla\Phi}}\|s_k\|^2, \dfrac{b}{L_{\nabla^2\Phi}}\|s_k\|\right\}$}{
        \State $t\gets t+1$;
        \State $y_{k, t} = y_{k, t-1} - \beta\nabla_y g(x_k, y_{k, t-1})$;
    }
    \State $N_k\gets t$;
    \State $y_{k+1} \gets y_{k, N_k}$;
    \State Obtain the approximate gradient and Hessian $\bm{G}_k = \widehat{\nabla}_x\Phi(x_k , y_{k+1}),\ \bm{H}_k = \widehat{\nabla}^2_{xx}\Phi(x_k, y_{k+1})$.
    \State Solve the CRN step \begin{align}
        s_{k+1} = \argmin_{s} \left\{\bm{G}_k^\top s + \dfrac{1}{2}s^\top \bm{H}_ks + \dfrac{M}{6}\|s\|^3 \right\}; \nonumber
    \end{align}
    \State $x_{k+1} = x_k + s_{k+1} $.
    }
    \State {\bf Output: $x_{\bar k}$, where $\bar k\in\arg\min_{2\le k\le K+1}
    \left\{
        \|s_{k}\|
    \right\}$.}
    
\end{algorithmic}
\end{algorithm}

In implementation, the cubic subproblem can be solved approximately by existing solvers, as we mentioned in Section~\ref{sec:literature}.  For example, \cite{carmon2019gradient} show that gradient descent can approximate the solution of the subproblem.

The complexity of the outer loop iteration of Algorithm \ref{alg:dl_crn_blo} is characterized in Proposition \ref{prop:conv_dl}. 
\begin{proposition}[Theorem 1 in \cite{wang2018note}]\label{prop:conv_dl}
Under Assumptions \ref{assumption:general1}, \ref{assumption:nonlinear2} and \ref{assumption:general2}, after $K$ outer iteration, there exist two universal constants such that, the sequence $\{x_k\}_{k=1}^K$ generated by Algorithm \ref{alg:dl_crn_blo} satisfies that there exists $k \in [1, \dots, K]$ such that\begin{align}
     \|\nabla \Phi(x_k)\| \le \dfrac{C_1}{(K-1)^{2/3}}, \qquad  \lambda_{\min}(\nabla^2 \Phi(x_k)) \ge -\dfrac{C_2}{(K-1)^{1/3}}. \nonumber
\end{align}
As a result, the number of outer loop iterations of Algorithm \ref{alg:dl_crn_blo} to obtain an $\varepsilon$-SOSP of $\Phi(x)$ is $K = \mathcal{O}(\varepsilon^{-1.5})$.
\end{proposition}

We now provide the lower level complexity analysis in Theorem \ref{thm:dl_ll}, accounting for the inner loop iterations. The proof of Theorem \ref{thm:dl_ll} can be found in Appendix \ref{sec:pf_of_dl}.

\begin{theorem}\label{thm:dl_ll}
    Under the same assumptions in Proposition \ref{prop:conv_dl}, for any given $\varepsilon$, if we run Algorithm \ref{alg:dl_crn_blo} with $N_k$ satisfying
    $\|\nabla_y g(x_k,y_{k,N_k})\|\le\mu_g
    \min\left\{\dfrac{a}{L_{\nabla\Phi}}\|s_k\|^2,\,\dfrac{b}{L_{\nabla^2\Phi}}\|s_k\|
    \right\}$ for $k\ge 1$, then the output of the algorithm is an $\varepsilon$-SOSP. The lower level iteration complexity is $\tilde{\mathcal O}(\varepsilon^{-1.5})$, i.e. we have \begin{align}
        \sum_{k=1}^{K} N_k \le \mathcal{O}(\varepsilon^{-1.5}\log(\varepsilon^{-1}))\nonumber.
    \end{align}
\end{theorem}

Despite its optimal outer loop iteration complexity, the strict nested-loop structure of DLCRN severely limits its practical applicability. Some applications with short response intervals at the lower level, for example in robotics, urban control, etc \citep{Duijkeren2015RealTimeNF,houska2011auto} fundamentally preclude solving the lower level problem with high precision at every outer iteration step. Because DLCRN requires repeated lower level subproblem solving to satisfy theoretical error bounds, it fails to avoid the extra $\mathcal O(\log(\varepsilon^{-1}))$ complexity of solving the lower level, and it obtains a suboptimal total convergence rate. Its suboptimal rate and incompatibility for some applications motivate the design of our Single Loop Cubic Regularized Newton (SLCRN) algorithm in the following section, which bypasses the inner loop complexity to achieve second order convergence guarantees using only a two-step lower level update.

\section{A Single Loop Cubic Regularized Newton Method for Deterministic Bilevel Optimization}\label{sec:4}

Section~\ref{sec:3} showed that a double loop implementation of Cubic Regularized Newton already yields a deterministic $\mathcal O(\varepsilon^{-1.5})$ outer loop convergence rate for the bilevel problem. The remaining issue is the lower level cost: DLCRN requires solving the lower level problem to a shrinking accuracy at every outer iteration, which leads to the extra logarithmic factor in the total lower level complexity. This motivates the search for a genuine single loop second order method.

To address this issue, we propose the Single Loop Cubic Regularized Newton method (SLCRN), which replaces the inner loop by exactly two steps of lower level update per outer iteration and hence preserves the deterministic optimal $\mathcal O(\varepsilon^{-1.5})$ second order convergence rate. 

\subsection{The SLCRN Algorithm}\label{sec:slcrn_alg}

Here we state the SLCRN algorithm. At each outer loop iteration, the lower level takes one gradient step and one Newton step on the lower level objective. However, the point obtained by Newton step is used only to compute the approximate hypergradient instead of updating the next lower level iteration point, while the gradient step point is used for the approximate hyper-Hessian and lower level variable update. See Algorithm~\ref{alg:sl_crn_blo} for details.

\begin{algorithm}[h]\label{alg:sl_crn_blo}
\caption{A Single Loop Cubic Regularized Newton algorithm for BLO problem (SLCRN)}
\begin{algorithmic}
    \State {\bf Input:} iteration number $K$, stepsize $\beta$, CRN parameter $M$
    \State {\bf Initialization:} initial point $x_1, y_1$
    \For{$k=1,\ldots,K$}{
        \State $y_{k+1}\gets y_k-\beta \nabla_y g(x_k,y_k);$
        \State $\hat y_{k+1}\gets y_{k+1}-\bigl(\nabla_{yy}^2 g(x_k,y_{k+1})\bigr)^{-1}\nabla_y g(x_k,y_{k+1});$
        \State Obtain the approximate gradient and Hessian:
        \begin{align*}
            \bm G_k=\widehat{\nabla}_x\Phi(x_k,\hat y_{k+1}),
            \qquad
            \bm H_k=\widehat{\nabla}_{xx}^2\Phi(x_k,y_{k+1}).
        \end{align*}
        \State Solve the CRN step:
        \begin{align*}
            s_{k+1}\in\argmin_s\left\{
            \bm G_k^\top s+\frac{1}{2}s^\top \bm H_k s+\frac{M}{6}\|s\|^3
            \right\};
        \end{align*}
        \State $x_{k+1}\gets x_k+s_{k+1}$.
    }
    \State {\bf Output:} $x_{\bar k + 1}$, where $\bar k\in\arg\min_{1\le k\le K}
    \left\{
        \|\nabla_y g(x_k,y_k)\|^3+\|s_{k+1}\|^3
    \right\}$.
\end{algorithmic}
\end{algorithm}

SLCRN has three main features. First, it is a \emph{single loop} method. Second, it preserves the \emph{optimal rate} of CRN in terms of $\varepsilon$. Third, the algorithmic structure is \emph{simple}. Each iteration consists of one lower level gradient step to obtain $y_{k+1}$, one lower level Newton step to obtain $\hat y_{k+1}$, and one upper level CRN step to obtain $x_{k+1}$, so the method is also relatively \emph{efficient to tune} in practice, compared with double loop approaches.

In addition, the asymmetry in Algorithm~\ref{alg:sl_crn_blo} is deliberate. The propagated lower level point is $y_{k+1}$, which is obtained by a single gradient step. Starting from $y_{k+1}$, we compute a Newton step $\hat y_{k+1}$, but this point is used only to construct the approximate hypergradient. In contrast, the approximate hyper-Hessian is still evaluated at the propagated point $y_{k+1}$. This reflects the asymmetric accuracy requirement of inexact CRN: the gradient estimate must be more accurate than the Hessian estimate. We discuss it in the following Section \ref{sec:gd_only}.

\subsection{Why We Need a Newton Step}\label{sec:gd_only}

The need for a Newton refinement is best understood by comparing the first order and second order single loop settings.
In the first order method \citep{hong2022twotimescaleframeworkbileveloptimization}, the propagated lower level gradient step is sufficient because the tracking error enters the analysis through a Lyapunov function $\Phi(x_k)+\tau\|y_k-y^*(x_k)\|^2$.
A single gradient step contracts the lower level error linearly, i.e.\begin{align*}
    \|\bm G_k-\nabla\Phi(x_k)\|
    \lesssim
    \|y_k-y^*(x_k)\|,
\end{align*} and this is enough to support first order hypergradient descent on the hyperfunction.
However, the picture changes for CRN. The upper level method is driven by a cubic model whose gradient part is especially sensitive to approximation error. The inexact CRN condition requires
\begin{align*}
    \|\bm G_k-\nabla\Phi(x_k)\|
    \lesssim
    \|s_k\|^2,
    \qquad
    \|\bm H_k-\nabla^2\Phi(x_k)\|
    \lesssim
    \|s_k\|,
\end{align*}
which suggests that: the lower level approximation used in the hypergradient for CRN must be more accurate than in first order methods to preserve the optimal convergence rate. For the single loop design, the accuracy requirement can be translated to the lower level tracking error,
\begin{align}
    \|\bm G_k-\nabla\Phi(x_k)\|
    \lesssim
    \|y_k-y^*(x_k)\|^2,
    \qquad
    \|\bm H_k-\nabla^2\Phi(x_k)\|
    \lesssim
    \|y_k-y^*(x_k)\|. \label{eq:4.2_1}
\end{align}
This is the key structural difference between single loop first order methods and single loop CRN. Based on this observation, solely applying a first order method for the lower level update is insufficient since it yields only linear convergence, not the quadratic convergence we need here. On the other hand, a Newton step yields local quadratic convergence, and the Hessian-vector product is already required inside the computation of hypergradient, so it is natural to directly employ a Newton step to meet the requirement in \eqref{eq:4.2_1}.

\subsection{Newton Alone Is Not Enough}\label{sec:newton}

The Newton step in Algorithm~\ref{alg:sl_crn_blo} is only an extra point computation and is not used to advance the lower level iterate. 
The reason is that the Newton method only guarantees local convergence, as discussed in the literature \citep{lin2007trust,di2019gradient}: it converges quadratically once the iterate is sufficiently close to the exact lower level solution, but it is not globally stable. If the Newton point is propagated as the next iteration update, then the lower level trajectory may leave the stable tracking regime, and the resulting lower level tracking error might diverge. Consequently, the hypergradient and hyper-Hessian may be formed from poor lower level information, and the upper level cubic model can become unreliable.

This phenomenon is illustrated by the following one-dimensional example. Consider the bilevel problem
\begin{align}
\min_{x\in\RR} \quad \Phi(x) &= f(x, y^*(x)) = \dfrac{1}{2}x^2 + \dfrac{1}{2}y^*(x)^2\\
    \text{s.t.}\quad y^*(x)&= \argmin_{y\in\RR} \quad g(x, y) = \dfrac{1}{2}(y-x)^2 + t\left[y\arctan(y) - \dfrac{1}{2}\log(1 + y^2)\right], \nonumber
\end{align}
where $t > 0$ is a fixed parameter. Since $\nabla_{yy}^2 g(x, y) = 1 + \dfrac{t}{1 + y^2} > 1$,
the lower level objective is strongly convex in $y$.
We compare the proposed SLCRN update with the following variant:
\begin{itemize}
\item \textbf{SLCRN:} proposed Algorithm \ref{alg:sl_crn_blo};
\item \textbf{Newton-update variant:} take the Newton step as the next iteration lower level update, i.e. use one gradient step and one Newton step to obtain $y_{k+1}$, then use one CRN step to obtain $x_{k+1}.$
\end{itemize}

Figure~\ref{fig:counter_2} plots the upper level value $f(x_k, y_k)$ and the lower level value $g(x_k, y_k)$ versus the outer iteration number. With the initial point $x_0=1$, $y_0=10$, and $t=10$, SLCRN converges rapidly to $(0,0)$, whereas the Newton-update variant diverges. Also, Figure~\ref{fig:plt_ll_43} gives a geometric view of the lower level dynamics of the first few iterations. In Figure~\ref{fig:plt_ll_43}, the Newton step $\hat y_2$ falls beyond the stable contraction area, and updating the lower level variable on it leads to a much larger lower level tracking error.

\begin{figure}[h!]
    \centering
    \includegraphics[width=0.75\linewidth]{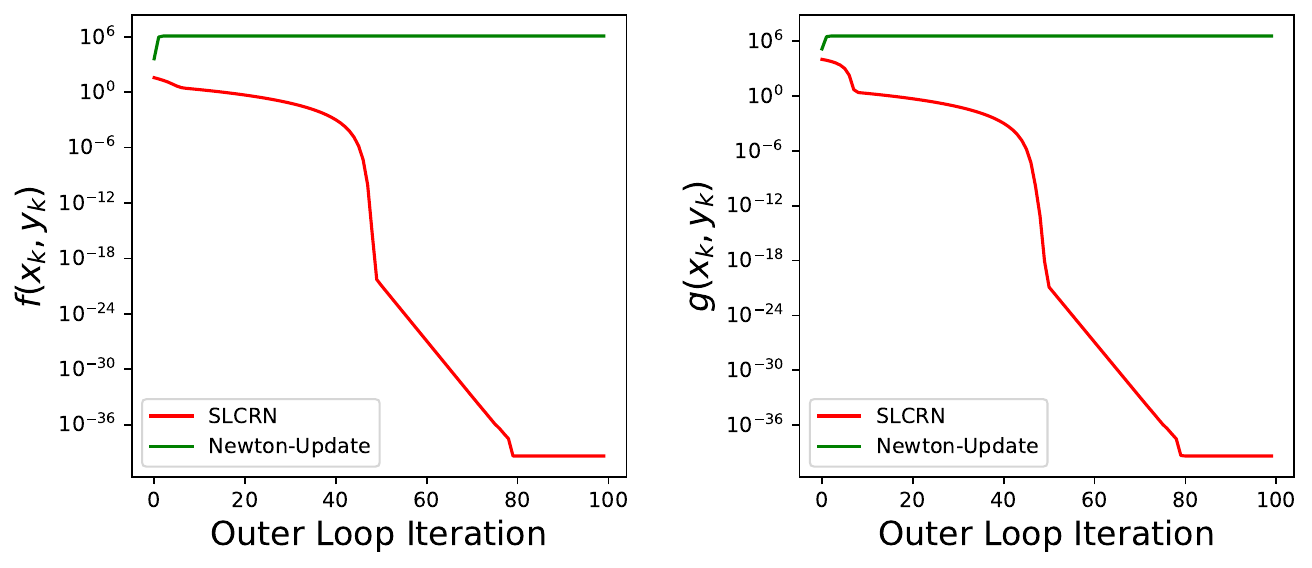}
    \caption{$f(x_k, y_k)$ (left) and $g(x_k, y_k)$ (right) versus the outer-loop iteration.}
    \label{fig:counter_2}
\end{figure}

\begin{figure}[h!]
    \centering
    \includegraphics[width=0.9\linewidth]{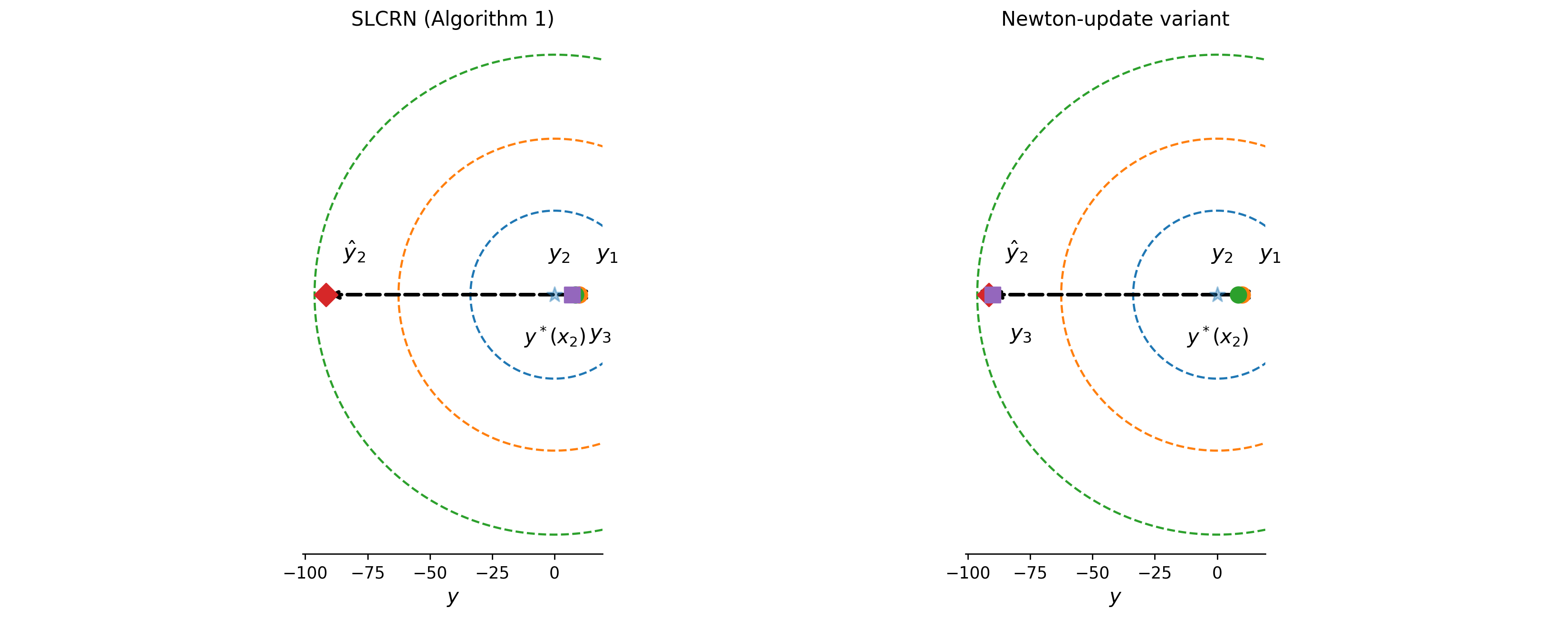}
    \caption{lower level updates of SLCRN and the Newton-update variant. On the right hand side, if we use the Newton point $\hat y_2$ to update, the next iteration lower level update $y_3$ has a much larger tracking error compared to the update scheme in Algorithm \ref{alg:sl_crn_blo}.}
    \label{fig:plt_ll_43}
\end{figure}

In summary, the architecture of SLCRN is explained in Sections~\ref{sec:gd_only} and \ref{sec:newton}. The lower level gradient step provides a stable single loop tracking contraction, and the Newton step is added because CRN requires a more accurate hypergradient than what gradient steps could provide. Finally, the Newton point is not used for updating because the Newton method could diverge and hence the Newton update scheme can be globally unstable. After the justifications in Sections~\ref{sec:gd_only} and \ref{sec:newton}, we are ready to present the main convergence result in Section~\ref{sec:slcrn_main}.

\subsection{Main Convergence Result}\label{sec:slcrn_main}

In this section, we provide a convergence analysis for SLCRN. In Theorem \ref{thm:main}, we formally present the oracle and iteration complexity of SLCRN in finding an $\varepsilon$-SOSP of $\Phi(x)$.

\begin{theorem}[Convergence of Algorithm \ref{alg:sl_crn_blo}]\label{thm:main}
Under Assumptions \ref{assumption:general1}, \ref{assumption:nonlinear2} and \ref{assumption:general2}, for Algorithm \ref{alg:sl_crn_blo}, if we set parameters as $\beta \le \dfrac{2}{\mu_g + \Lgg}$ and $M > M_0$, where $M_0$ is the right hand side of equation \cref{eq:satif_M}, the number of iterations required for Algorithm \ref{alg:sl_crn_blo} to obtain an $\varepsilon$-SOSP of $\Phi(x)$ is
        $ K= \mathcal{O}(\varepsilon^{-1.5})$. 
\end{theorem}

The proof of Theorem \ref{thm:main} is relegated to Appendix \ref{sec:pf_of_main}. The argument combines the two roles of the lower level updates described above: the gradient step preserves a stable single loop tracking process, while the Newton step sharpens the lower level approximation at the point where the hypergradient is evaluated. Together, these updates yield an optimal deterministic $\mathcal O(\varepsilon^{-1.5})$ SOSP convergence rate without the extra logarithmic lower level cost of repeated inner solves.

\section{Numerical Experiments}

\nocite{chen2021cubic,luo2022finding}

In the previous section, we presented the convergence analysis of SLCRN and justify that some simple variants fail to guarantee the optimal convergence rate. In this section, we conclude with a numerical experiment that evaluates the efficiency of SLCRN in finding SOSPs compared to previous works. Following \cite{huang2025efficiently}, we carry out our experiment on the $d$–dimensional nonconvex–strongly-convex test function. More precisely, we consider \begin{align}
    \min_{x\in\RR^d}\quad \Phi(x) &:= f(x, y^*(x)) \nonumber\\
    \text{s.t.}\quad y^*(x) &= \argmin_{y\in\RR} g(x, y).\nonumber
\end{align}
The upper level function is defined as a piecewise function \begin{align}
f(x,y)=
\begin{cases}
f_{i,1}(x,y) &x_{1},\dots,x_{i-1}\in[2\tau,6\tau],x_{i}\in[0,\tau],x_{i+1},\dots,x_{d}\in[0,\tau],1\le i\le d-1;\\
f_{i,2}(x,y) &x_{1},\dots,x_{i-1}\in[2\tau,6\tau],x_{i}\in[\tau,2\tau],x_{i+1},\dots,x_{d}\in[0,\tau],1\le i\le d-1;\\
f_{d,1}(x,y) &x_{1},\dots,x_{d-1}\in[2\tau,6\tau], x_{d}\in[0,\tau]; \\
f_{d,2}(x,y) &x_{1},\dots,x_{d-1}\in[2\tau,6\tau],x_{d}\in[\tau,2\tau];\\
f_{d+1,1}(x,y) &x_{1},\dots,x_{d}\in[2\tau,6\tau],
\end{cases}\nonumber
\end{align}
where \begin{align}
f_{i,1}(x,y) &=\sum_{j=1}^{\,i-1} L\bigl(x_j-4\tau\bigr)^{2}-\gamma x_{i}^{2}+\sum_{j=i+1}^{d} L x_{j}^{2}-(i-1)\nu, & 1 \le i \le d-1, \nonumber\\
f_{i,2}(x,y) &=
\sum_{j=1}^{\,i-1} L\bigl(x_j-4\tau\bigr)^{2}+y+\sum_{j=i+2}^{d} L x_{j}^{2}
-(i-1)\nu,
& 1 \le i \le d-1, \nonumber\\
f_{d,1}(x,y) &=
\sum_{j=1}^{d-1} L\bigl(x_j-4\tau\bigr)^{2}
- \gamma x_{d}^{2}
- (d-1)\nu, \nonumber\\
f_{d,2}(x,y) &=
\sum_{j=1}^{\,d-1} L\bigl(x_j-4\tau\bigr)^{2}
\;+\;y
\;-\;(d-1)\nu, \nonumber\\
f_{d+1,1}(x,y) &=
\sum_{j=1}^{\,d} L\bigl(x_j-4\tau\bigr)^{2}
- d\nu,\nonumber
\end{align}
and the lower level problem is \begin{align}
    g(x, y) = \dfrac{y^2}{2} - h(x)y,\nonumber
\end{align}
where \begin{align}
g(x)&=\begin{cases}
h_{1}(x_{i}) + h_{2}(x_{i})\,x_{i+1}^{2}, &x_{1},\dots ,x_{i-1}\in[2\tau,6\tau],x_{i}\in[\tau,2\tau],x_{i+1},\dots,x_{d}\in[0,\tau],1\le i\le d-1;\nonumber\\
h_{1}(x_{d}), &x_{1},\dots ,x_{d-1}\in[2\tau,6\tau], x_{d}\in[\tau,2\tau];\nonumber\\
0, & \text{elsewhere},
\end{cases}\\
h_{1}(c) &= -\gamma c^{2}
          + \frac{(-14L + 10\gamma)(c-\tau)^{3}}{3\tau}
          + \frac{(5L - 3\gamma)(c-\tau)^{4}}{2\tau^{2}},\nonumber\\
h_{2}(c) &= -\gamma
          - \frac{10(L+\gamma)(c-2\tau)^{3}}{\tau^{3}}
          - \frac{15(L+\gamma)(c-2\tau)^{4}}{\tau^{4}}
          - \frac{ 6(L+\gamma)(c-2\tau)^{5}}{\tau^{5}},\nonumber
\end{align}
and the constants satisfy \begin{align}
    L > 0, \quad\gamma > 0, \quad\tau = e, \quad \nu = -h_1(2\tau) +4L\tau^2.\nonumber
\end{align}
The landscape of $\Phi$ contains $d$ strict saddles\begin{align}
    (0,\dots,0)^\top, \ (4\tau, 0, \dots, \tau)^\top, \ \cdots\ , \ (4\tau, \dots, 4\tau, 0)^\top,\nonumber
\end{align}
but a single global minimizer at $(4\tau,\ldots,4\tau)^{\top}$.

We test the efficiency of SLCRN on the problem above and compare it with several existing algorithms, including AID–BiO \citep{ji2021bilevel}, PBO \citep{huang2025efficiently}, iNEON \citep{huang2025efficiently}), PRAHGD \citep{yang2023accelerating}.
We run these algorithms on the above problem and denote AID-BiO in \cite{ji2021bilevel} by BO in the numerical plot. We specify the hyperparameters and settings of these algorithms as follows:

\textbf{SLCRN}:  
The lower level gradient has a fixed step size
$\beta=1$; the cubic regularization constant is $M=100$.

\textbf{PBO}: 
We follow the settings recommended by the authors: lower level stepsize
$\beta=1$ and perturbation radius $r=0.01$, inner loop iteration number $\mathcal{T}=5$, outer hyper-gradient descent
stepsize $0.05$.

\textbf{AID–BiO (BO in the figure)}:  
We use the same lower level step size $\beta=1$, inner loop iteration number $\mathcal{T}=5$, outer hyper-gradient descent
stepsize also $0.05$.

\textbf{iNEON}:
Identical lower level stepsize $\beta=1$, inner loop iteration number $\mathcal{T}=5$, outer hyper-gradient descent stepsize also $0.05$. The iNEON subroutine is triggered when the hyper-gradient norm is below $10^{-3}$, and we set the iNEON iteration number to $30$, the update stepsize to $0.1$, the initialization radius to $10^{-3}$, the maximum search radius to $5\times10^{-2}$, the curvature threshold to $10^{-4}$, and the negative-curvature step size to $5\times10^{-2}$.

\textbf{PRAHGD}:
The outer accelerated hyper-gradient update has stepsize $0.05$. We set the momentum parameter to $0.5$, the restart threshold parameter to $0.5$, the perturbation radius to $10^{-2}$, and the perturbation is triggered only at restart when the hyper-gradient norm is below $10^{-3}$.

All methods start from the same point
$x_{0}\in\mathcal{N}(0, I_d),y_{0}=1$ and run for $K=1000$ outer iterations. In Figure \ref{fig:plot_escape} we plot the outer loop function value gap $f(x_k, y_k) - \Phi^*$ versus the outer loop iteration. The results show that our method is faster to escape the saddle point and locate the SOSP of the hyperfunction than PBO and BO.
\begin{figure}[h]
    \centering
    \includegraphics[width=\linewidth]{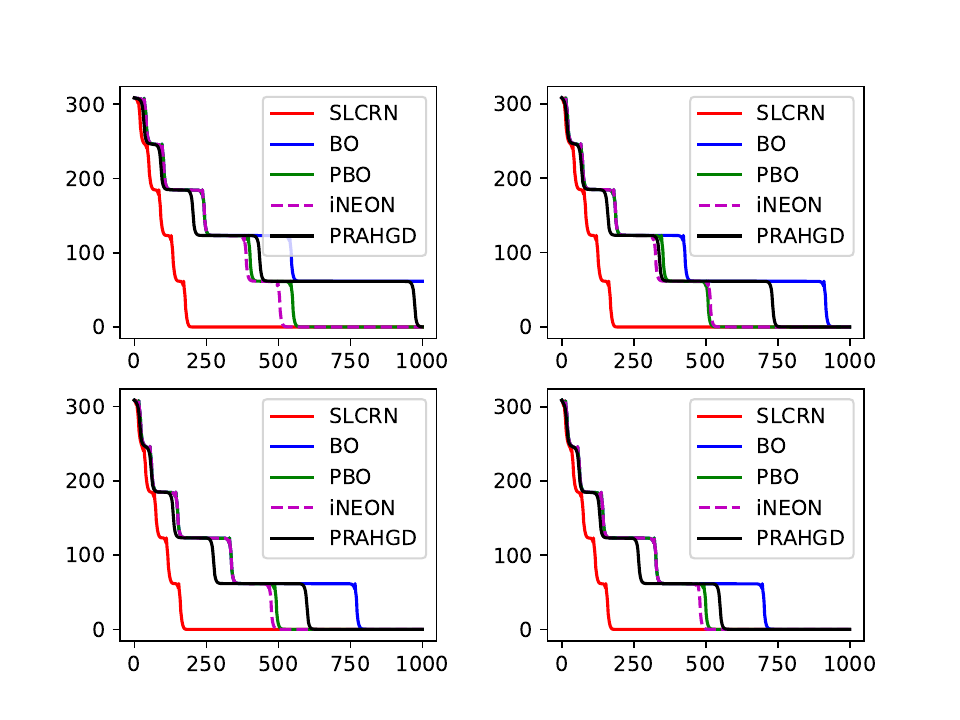}
    \caption{Comparison between Algorithm \ref{alg:sl_crn_blo} and PBO, BO. We start from 4 different initial points $x_0$ in each replication. The x-axis represents the outer loop iteration number, and the y-axis represents the upper level function value $f(x_k, y_k) - \Phi^*$.}
    \label{fig:plot_escape}
\end{figure}

\newpage
\bibliographystyle{abbrvnat} 
\bibliography{reference}

@article{luo2022finding,
  title={Finding Second-Order Stationary Points in Nonconvex-Strongly-Concave Minimax Optimization},
  author={Luo, Luo and Li, Yujun and Chen, Cheng},
  journal={Advances in Neural Information Processing Systems},
  volume={35},
  pages={36667--36679},
  year={2022}
}

@article{article,
author = {Klaricic Bakula, Milica},
year = {2018},
month = {11},
pages = {},
title = {Jensen–Steffensen inequality for strongly convex functions},
volume = {2018},
journal = {Journal of Inequalities and Applications},
doi = {10.1186/s13660-018-1897-2}
}

@article{labbe2016bilevel,
  title={Bilevel programming and price setting problems},
  author={Labb{\'e}, Martine and Violin, Alessia},
  journal={Annals of operations research},
  volume={240},
  pages={141--169},
  year={2016},
  publisher={Springer}
}

@article{zeng2024demonstrations,
  title={When demonstrations meet generative world models: A maximum likelihood framework for offline inverse reinforcement learning},
  author={Zeng, Siliang and Li, Chenliang and Garcia, Alfredo and Hong, Mingyi},
  journal={Advances in Neural Information Processing Systems},
  volume={36},
  year={2024}
}

@inproceedings{ji2021bilevel,
  title={Bilevel optimization: Convergence analysis and enhanced design},
  author={Ji, Kaiyi and Yang, Junjie and Liang, Yingbin},
  booktitle={International conference on machine learning},
  pages={4882--4892},
  year={2021},
  organization={PMLR}
}

@misc{shen2023penaltybasedbilevelgradientdescent,
      title={On Penalty-based Bilevel Gradient Descent Method}, 
      author={Han Shen and Quan Xiao and Tianyi Chen},
      year={2023},
      eprint={2302.05185},
      archivePrefix={arXiv},
      primaryClass={cs.LG},
      url={https://arxiv.org/abs/2302.05185}, 
}

@misc{ghadimi2018approximationmethodsbilevelprogramming,
      title={Approximation Methods for Bilevel Programming}, 
      author={Saeed Ghadimi and Mengdi Wang},
      year={2018},
      eprint={1802.02246},
      archivePrefix={arXiv},
      primaryClass={math.OC},
      url={https://arxiv.org/abs/1802.02246}, 
}

@article{hong2022twotimescaleframeworkbileveloptimization,
  title={A two-timescale stochastic algorithm framework for bilevel optimization: Complexity analysis and application to actor-critic},
  author={Hong, Mingyi and Wai, Hoi-To and Wang, Zhaoran and Yang, Zhuoran},
  journal={SIAM Journal on Optimization},
  volume={33},
  number={1},
  pages={147--180},
  year={2023},
  publisher={SIAM}
}

@inproceedings{franceschi2017forwardreversegradientbasedhyperparameter,
  title={Forward and Reverse Gradient-Based Hyperparameter Optimization},
  author={Luca Franceschi and Michele Donini and Paolo Frasconi and Massimiliano Pontil},
  booktitle={International Conference on Machine Learning},
  year={2017},
  url={https://api.semanticscholar.org/CorpusID:8026824}
}

@article{chen2022singletimescalemethodstochasticbilevel,
  title={A Single-Timescale Stochastic Bilevel Optimization Method},
  author={Tianyi Chen and Yuejiao Sun and Wotao Yin},
  journal={ArXiv},
  year={2021},
  volume={abs/2102.04671},
  url={https://api.semanticscholar.org/CorpusID:231855754}
}

@misc{lu2024firstorderpenaltymethodsbilevel,
      title={First-order penalty methods for bilevel optimization}, 
      author={Zhaosong Lu and Sanyou Mei},
      year={2024},
      eprint={2301.01716},
      archivePrefix={arXiv},
      primaryClass={math.OC},
      url={https://arxiv.org/abs/2301.01716}, 
}

@article{Silvrio2022ABO,
  title={A bi-objective optimization model for segment routing traffic engineering},
  author={Antonio Jos{\'e} Silv{\'e}rio and Rodrigo De Souza Couto and Miguel Elias M. Campista and Lu{\'i}s Henrique Maciel Kosmalski Costa},
  journal={Annals of Telecommunications},
  year={2022},
  volume={77},
  pages={813 - 824},
  url={https://api.semanticscholar.org/CorpusID:247004593}
}

@article{wang2018note,
  title={A Note on Inexact Condition for Cubic Regularized Newton's Method},
  author={Wang, Zhe and Zhou, Yi and Liang, Yingbin and Lan, Guanghui},
  journal={arXiv preprint arXiv:1808.07384},
  year={2018}
}

@book{nesterov2013introductory,
  title={Introductory lectures on convex optimization: A basic course},
  author={Nesterov, Yurii},
  volume={87},
  year={2013},
  publisher={Springer Science \& Business Media}
}

@article{jiang2024barrier,
  title={Barrier Function for Bilevel Optimization with Coupled Lower-Level Constraints: Formulation, Approximation and Algorithms},
  author={Jiang, Xiaotian and Li, Jiaxiang and Hong, Mingyi and Zhang, Shuzhong},
  journal={arXiv preprint arXiv:2410.10670},
  year={2024}
}

@article{nesterov2006cubic,
  title={Cubic regularization of Newton method and its global performance},
  author={Nesterov, Yurii and Polyak, Boris T},
  journal={Mathematical programming},
  volume={108},
  number={1},
  pages={177--205},
  year={2006},
  publisher={Springer}
}

@article{chayti2024improving,
  title={Improving Stochastic Cubic Newton with Momentum},
  author={Chayti, El Mahdi and Doikov, Nikita and Jaggi, Martin},
  journal={arXiv preprint arXiv:2410.19644},
  year={2024}
}

@article{chen2021cubic,
  title={A cubic regularization approach for finding local minimax points in nonconvex minimax optimization},
  author={Chen, Ziyi and Hu, Zhengyang and Li, Qunwei and Wang, Zhe and Zhou, Yi},
  journal={arXiv preprint arXiv:2110.07098},
  year={2021}
}

@article{huang2025efficiently,
  title={Efficiently escaping saddle points in bilevel optimization},
  author={Huang, Minhui and Chen, Xuxing and Ji, Kaiyi and Ma, Shiqian and Lai, Lifeng},
  journal={Journal of Machine Learning Research},
  volume={26},
  number={1},
  pages={1--61},
  year={2025}
}

@inproceedings{jiang2024krylov,
  title={Krylov Cubic Regularized Newton: A Subspace Second-Order Method with Dimension-Free Convergence Rate},
  author={Jiang, Ruichen and Raman, Parameswaran and Sabach, Shoham and Mokhtari, Aryan and Hong, Mingyi and Cevher, Volkan},
  booktitle={International Conference on Artificial Intelligence and Statistics},
  pages={4411--4419},
  year={2024},
  organization={PMLR}
}

@article{carmon2019gradient,
  title={Gradient descent finds the cubic-regularized nonconvex Newton step},
  author={Carmon, Yair and Duchi, John},
  journal={SIAM Journal on Optimization},
  volume={29},
  number={3},
  pages={2146--2178},
  year={2019},
  publisher={SIAM}
}

@article{nesterov2008accelerating,
  title={Accelerating the cubic regularization of Newton’s method on convex problems},
  author={Nesterov, Yu},
  journal={Mathematical Programming},
  volume={112},
  number={1},
  pages={159--181},
  year={2008},
  publisher={Springer}
}

@article{cartis2011adaptive,
  title={Adaptive cubic regularisation methods for unconstrained optimization. Part I: motivation, convergence and numerical results},
  author={Cartis, Coralia and Gould, Nicholas IM and Toint, Philippe L},
  journal={Mathematical Programming},
  volume={127},
  number={2},
  pages={245--295},
  year={2011},
  publisher={Springer}
}

@article{ghadimi2017second,
  title={Second-order methods with cubic regularization under inexact information},
  author={Ghadimi, Saeed and Liu, Han and Zhang, Tong},
  journal={arXiv preprint arXiv:1710.05782},
  year={2017}
}

@article{cartis2011adaptive2,
  title={Adaptive cubic regularisation methods for unconstrained optimization. Part II: worst-case function-and derivative-evaluation complexity},
  author={Cartis, Coralia and Gould, Nicholas IM and Toint, Philippe L},
  journal={Mathematical programming},
  volume={130},
  number={2},
  pages={295--319},
  year={2011},
  publisher={Springer}
}

@article{tripuraneni2018stochastic,
  title={Stochastic cubic regularization for fast nonconvex optimization},
  author={Tripuraneni, Nilesh and Stern, Mitchell and Jin, Chi and Regier, Jeffrey and Jordan, Michael I},
  journal={Advances in neural information processing systems},
  volume={31},
  year={2018}
}

@inproceedings{franceschi2018bilevel,
  title={Bilevel programming for hyperparameter optimization and meta-learning},
  author={Franceschi, Luca and Frasconi, Paolo and Salzo, Saverio and Grazzi, Riccardo and Pontil, Massimiliano},
  booktitle={International conference on machine learning},
  pages={1568--1577},
  year={2018},
  organization={PMLR}
}

@inproceedings{ge2015escaping,
  title={Escaping from saddle points—online stochastic gradient for tensor decomposition},
  author={Ge, Rong and Huang, Furong and Jin, Chi and Yuan, Yang},
  booktitle={Conference on learning theory},
  pages={797--842},
  year={2015},
  organization={PMLR}
}

@inproceedings{jin2017escape,
  title={How to escape saddle points efficiently},
  author={Jin, Chi and Ge, Rong and Netrapalli, Praneeth and Kakade, Sham M and Jordan, Michael I},
  booktitle={International conference on machine learning},
  pages={1724--1732},
  year={2017},
  organization={PMLR}
}

@article{jin2021nonconvex,
  title={On nonconvex optimization for machine learning: Gradients, stochasticity, and saddle points},
  author={Jin, Chi and Netrapalli, Praneeth and Ge, Rong and Kakade, Sham M and Jordan, Michael I},
  journal={Journal of the ACM (JACM)},
  volume={68},
  number={2},
  pages={1--29},
  year={2021},
  publisher={ACM New York, NY, USA}
}

@inproceedings{fang2019sharp,
  title={Sharp analysis for nonconvex sgd escaping from saddle points},
  author={Fang, Cong and Lin, Zhouchen and Zhang, Tong},
  booktitle={Conference on Learning Theory},
  pages={1192--1234},
  year={2019},
  organization={PMLR}
}

@inproceedings{finn2017model,
  title={Model-agnostic meta-learning for fast adaptation of deep networks},
  author={Finn, Chelsea and Abbeel, Pieter and Levine, Sergey},
  booktitle={International conference on machine learning},
  pages={1126--1135},
  year={2017},
  organization={PMLR}
}

@inproceedings{maclaurin2015gradient,
  title={Gradient-based hyperparameter optimization through reversible learning},
  author={Maclaurin, Dougal and Duvenaud, David and Adams, Ryan},
  booktitle={International conference on machine learning},
  pages={2113--2122},
  year={2015},
  organization={PMLR}
}

@article{colson2007overview,
  title={An overview of bilevel optimization},
  author={Colson, Beno{\^\i}t and Marcotte, Patrice and Savard, Gilles},
  journal={Annals of operations research},
  volume={153},
  number={1},
  pages={235--256},
  year={2007},
  publisher={Springer}
}

@article{colson2005bilevel,
  title={Bilevel programming: A survey},
  author={Colson, Beno{\^\i}t and Marcotte, Patrice and Savard, Gilles},
  journal={4or},
  volume={3},
  number={2},
  pages={87--107},
  year={2005},
  publisher={Springer}
}

@article{brotcorne2001bilevel,
  title={A bilevel model for toll optimization on a multicommodity transportation network},
  author={Brotcorne, Luce and Labb{\'e}, Martine and Marcotte, Patrice and Savard, Gilles},
  journal={Transportation science},
  volume={35},
  number={4},
  pages={345--358},
  year={2001},
  publisher={INFORMS}
}

@article{reisizadeh2025blur,
  title={BLUR: A Bi-Level Optimization Approach for LLM Unlearning},
  author={Reisizadeh, Hadi and Jia, Jinghan and Bu, Zhiqi and Vinzamuri, Bhanukiran and Ramakrishna, Anil and Chang, Kai-Wei and Cevher, Volkan and Liu, Sijia and Hong, Mingyi},
  journal={arXiv preprint arXiv:2506.08164},
  year={2025}
}

@inproceedings{fan2024challenging,
  title={Challenging forgets: Unveiling the worst-case forget sets in machine unlearning},
  author={Fan, Chongyu and Liu, Jiancheng and Hero, Alfred and Liu, Sijia},
  booktitle={European Conference on Computer Vision},
  pages={278--297},
  year={2024},
  organization={Springer}
}

@article{chakraborty2023parl,
  title={PARL: A unified framework for policy alignment in reinforcement learning from human feedback},
  author={Chakraborty, Souradip and Bedi, Amrit Singh and Koppel, Alec and Manocha, Dinesh and Wang, Huazheng and Wang, Mengdi and Huang, Furong},
  journal={arXiv preprint arXiv:2308.02585},
  year={2023}
}

@article{shen2025penalty,
  title={On penalty-based bilevel gradient descent method},
  author={Shen, Han and Xiao, Quan and Chen, Tianyi},
  journal={Mathematical Programming},
  pages={1--51},
  year={2025},
  publisher={Springer}
}

@article{kwon2023penalty,
  title={On penalty methods for nonconvex bilevel optimization and first-order stochastic approximation},
  author={Kwon, Jeongyeol and Kwon, Dohyun and Wright, Stephen and Nowak, Robert},
  journal={arXiv preprint arXiv:2309.01753},
  year={2023}
}

@inproceedings{chen2024finding,
  title={On finding small hyper-gradients in bilevel optimization: Hardness results and improved analysis},
  author={Chen, Lesi and Xu, Jing and Zhang, Jingzhao},
  booktitle={The Thirty Seventh Annual Conference on Learning Theory},
  pages={947--980},
  year={2024},
  organization={PMLR}
}

@article{chen2025set,
  title={Set Smoothness Unlocks Clarke Hyper-stationarity in Bilevel Optimization},
  author={Chen, He and Li, Jiajin and So, Anthony Man-cho},
  journal={arXiv preprint arXiv:2506.04587},
  year={2025}
}

@article{jiang2025correspondence,
  title={A Correspondence-Driven Approach for Bilevel Decision-making with Nonconvex Lower-Level Problems},
  author={Jiang, Xiaotian and Li, Jiaxiang and Hong, Mingyi and Zhang, Shuzhong},
  journal={arXiv preprint arXiv:2509.01148},
  year={2025}
}

@article{bolte2025bilevel,
  title={Bilevel gradient methods and Morse parametric qualification},
  author={Bolte, J{\'e}r{\^o}me and Le, Quoc-Tung and Pauwels, Edouard and Vaiter, Samuel},
  journal={arXiv preprint arXiv:2502.09074},
  year={2025}
}

@article{xiao2024unlocking,
  title={Unlocking Global Optimality in Bilevel Optimization: A Pilot Study},
  author={Xiao, Quan and Chen, Tianyi},
  journal={arXiv preprint arXiv:2408.16087},
  year={2024}
}

@article{jiang2025discretization,
  title={A Discretization Approach for Bilevel Optimization with Low-Dimensional and Non-Convex Lower-Level},
  author={Jiang, Xiaotian and Tsaknakis, Ioannis and Khanduri, Prashant and Hong, Mingyi},
  journal={arXiv preprint arXiv:2505.10830},
  year={2025}
}

@article{liu2022bome,
  title={Bome! bilevel optimization made easy: A simple first-order approach},
  author={Liu, Bo and Ye, Mao and Wright, Stephen and Stone, Peter and Liu, Qiang},
  journal={Advances in neural information processing systems},
  volume={35},
  pages={17248--17262},
  year={2022}
}

@article{zhang2022drsom,
  title={DRSOM: A dimension reduced second-order method and preliminary analyses},
  author={Zhang, Chuwen and Ge, Dongdong and Jiang, Bo and Ye, Yinyu},
  journal={arXiv preprint arXiv:2208.00208},
  year={2022}
}

@article{he2025homogeneous,
  title={Homogeneous second-order descent framework: a fast alternative to Newton-type methods},
  author={He, Chang and Jiang, Yuntian and Zhang, Chuwen and Ge, Dongdong and Jiang, Bo and Ye, Yinyu},
  journal={Mathematical Programming},
  pages={1--62},
  year={2025},
  publisher={Springer}
}

@article{zhang2025homogeneous,
  title={A Homogeneous Second-Order Descent Method for Nonconvex Optimization},
  author={Zhang, Chuwen and He, Chang and Jiang, Yuntian and Xue, Chenyu and Jiang, Bo and Ge, Dongdong and Ye, Yinyu},
  journal={Mathematics of Operations Research},
  year={2025},
  publisher={INFORMS}
}

@article{jiang2023beyond,
  title={Beyond Nonconvexity: A Universal Trust-Region Method with New Analyses},
  author={Jiang, Yuntian and He, Chang and Zhang, Chuwen and Ge, Dongdong and Jiang, Bo and Ye, Yinyu},
  journal={arXiv e-prints},
  pages={arXiv--2311},
  year={2023}
}

@inproceedings{yin2022bm,
  title={Bm-nas: Bilevel multimodal neural architecture search},
  author={Yin, Yihang and Huang, Siyu and Zhang, Xiang},
  booktitle={Proceedings of the AAAI Conference on Artificial Intelligence},
  volume={36},
  number={8},
  pages={8901--8909},
  year={2022}
}

@article{tu2024efficient,
  title={Efficient architecture search via bi-level data pruning},
  author={Tu, Chongjun and Ye, Peng and Lin, Weihao and Ye, Hancheng and Yu, Chong and Chen, Tao and Li, Baopu and Ouyang, Wanli},
  journal={IEEE Transactions on Circuits and Systems for Video Technology},
  year={2024},
  publisher={IEEE}
}

@article{huang2022cubic,
  title={Cubic regularized Newton method for the saddle point models: A global and local convergence analysis},
  author={Huang, Kevin and Zhang, Junyu and Zhang, Shuzhong},
  journal={Journal of Scientific Computing},
  volume={91},
  number={2},
  pages={60},
  year={2022},
  publisher={Springer}
}

@article{huang2025approximation,
  title={An approximation-based regularized extra-gradient method for monotone variational inequalities},
  author={Huang, Kevin and Zhang, Shuzhong},
  journal={SIAM Journal on Optimization},
  volume={35},
  number={3},
  pages={1469--1497},
  year={2025},
  publisher={SIAM}
}

@article{xian2025escaping,
  title={Escaping saddle point efficiently in minimax and bilevel optimizations},
  author={Xian, Wenhan and Huang, Feihu and Huang, Heng},
  year={2025}
}

@article{yang2023accelerating,
  title={Accelerating inexact hypergradient descent for bilevel optimization},
  author={Yang, Haikuo and Luo, Luo and Li, Chris Junchi and Jordan, Michael I},
  journal={arXiv preprint arXiv:2307.00126},
  year={2023}
}

@article{Duijkeren2015RealTimeNF,
  title={Real-Time NMPC for Semi-Automated Highway Driving of Long Heavy Vehicle Combinations},
  author={Niels van Duijkeren and Tam{\'a}s Keviczky and Peter Nilsson and Leo Laine},
  journal={IFAC-PapersOnLine},
  year={2015},
  volume={48},
  pages={39-46},
  url={https://api.semanticscholar.org/CorpusID:55844116}
}

@article{houska2011auto,
  title={An auto-generated real-time iteration algorithm for nonlinear MPC in the microsecond range},
  author={Houska, Boris and Ferreau, Hans Joachim and Diehl, Moritz},
  journal={Automatica},
  volume={47},
  number={10},
  pages={2279--2285},
  year={2011},
  publisher={Elsevier}
}

@article{he2025history,
  title={History-aware adaptive high-order tensor regularization},
  author={He, Chang and Jiang, Bo and Jiang, Yuntian and Zhang, Chuwen and Zhang, Shuzhong},
  journal={arXiv preprint arXiv:2511.05788},
  year={2025}
}

@inproceedings{lin2007trust,
  title={Trust region newton methods for large-scale logistic regression},
  author={Lin, Chih-Jen and Weng, Ruby C and Keerthi, S Sathiya},
  booktitle={Proceedings of the 24th international conference on Machine learning},
  pages={561--568},
  year={2007}
}

@inproceedings{di2019gradient,
  title={A gradient-based globalization strategy for the Newton method},
  author={di Serafino, Daniela and Toraldo, Gerardo and Viola, Marco},
  booktitle={International Conference on Numerical Computations: Theory and Algorithms},
  pages={177--185},
  year={2019},
  organization={Springer}
}

@article{yang2026second,
  title={Second-Order Bilevel Optimization with Accelerated Convergence Rates},
  author={Yang, Sheng and Liu, Chengchang and Chen, Lesi and Lui, John},
  journal={arXiv preprint arXiv:2605.06431},
  year={2026}
}

\appendix

\section{Proof of Results in Section \ref{sec:2}}

\subsection{Proof of Lemma \ref{lem:lip_xy}}\label{sec:pf_lip_xy}
\begin{proof}
    Let $H_1 = \nabla_{yy}^2g(x_1, y_1), H_2 = \nabla_{yy}^2g(x_2, y_2)$, and let \begin{align}
        \Delta_1 &= \|\nabla^2_{xx}f(x_1, y_1) - \nabla^2_{xx}f(x_2, y_2)\|,\nonumber\\
        \Delta_2 &= \|\nabla^2_{xy}g(x_1, y_1)H_1^{-1}\nabla^2_{xy}f(x_1, y_1) - \nabla^2_{xy}g(x_2, y_2)H_2^{-1}\nabla^2_{xy}f(x_2, y_2)\|,\nonumber\\
        \Delta_3 &= \|\nabla^2_{xy}g(x_1, y_1)H_1^{-1}\left(\nabla^2_{xy}f(x_1, y_1) - \nabla^2_{xy}g(x_1, y_1)H_1^{-1}\nabla^2_{yy}f(x_1, y_1)\right)\notag\\&\quad - \nabla^2_{xy}g(x_2, y_2)H_2^{-1}\left(\nabla^2_{xy}f(x_2, y_2) - \nabla^2_{xy}g(x_2, y_2)H_2^{-1}\nabla^2_{yy}f(x_2, y_2)\right)\|,\nonumber\\
        \Delta_4 &= \|\nabla^2_{xy}g(x_1, y_1)\left(H_1^{-1}\nabla_{yxy} g(x_1, y_1)H_1^{-1}\right)\nabla_yf(x_1, y_1)\nonumber\\&\quad - \nabla^2_{xy}g(x_2, y_2)\left(H_2^{-1}\nabla_{yxy} g(x_2, y_2)H_2^{-1}\right)\nabla_yf(x_2, y_2)\|,\nonumber\\
        \Delta_5 &= \|\left(\nabla^3_{xxy}g(x_1, y_1) -\nabla^2_{xy}g(x_1,y_1)H_1^{-1}\nabla^3_{yxy}g(x_1, y_1)\right)H_1^{-1}\nabla_yf(x_1, y_1) \notag\\
        &\quad- \left(\nabla^3_{xxy}g(x_2, y_2) -\nabla^2_{xy}g(x_2,y_2))H_2^{-1}\nabla^3_{yxy}g(x_2, y_2)\right)H_2^{-1}\nabla_yf(x_2, y_2)\|.\nonumber
    \end{align}
    From Assumption \ref{assumption:general2} and \ref{assumption:nonlinear2} we know that \begin{align}
        \|\left(\nabla_{yy}^2g(x, y)\right)^{-1}\|&\le \mu_g^{-1}, \nonumber\\
        \|H_1^{-1} - H_2^{-1}\| &= \|H_1^{-1}(H_2 - H_1)H_2^{-1}\| \le \dfrac{\|H_2 - H_1\|}{\|H_1H_2\|} \le \dfrac{L_{g, 3}}{\mu_g^2},\nonumber\\
        \left\|\dfrac{\partial H_1}{\partial x} - \dfrac{\partial H_2}{\partial x}\right\|&\le \|\partial_xH(x_1, y_1) - \partial_xH(x_2, y_2)\| + \|\partial_xy_1\partial_yH(x_1, y_1) - \partial_xy_2\partial_yH(x_2, y_2)\|\nonumber\\
        &= \|\nabla^3_{xyy}g(x_1, y_1) - \nabla^3_{xyy}g(x_2, y_2)\|\notag\\
        &+\|\nabla^2_{xy}g(x_1,y_1) H^{-1}_1\nabla^3_{yyy}g(x_1, y_1) - \nabla^2_{xy}g(x_2,y_2) H^{-1}_2\nabla^3_{yyy}g(x_2, y_2)\|\notag\\
        &\le \|\nabla^3_{xyy}g(x_1, y_1) - \nabla^3_{xyy}g(x_2, y_2)\|\notag\\
        &+\|\nabla^2_{xy}g(x_1,y_1) - \nabla^2_{xy}g(x_2,y_2)\|\|H^{-1}_1\nabla^3_{yyy}g(x_1, y_1)\|\notag\\
        &+\|\nabla^2_{xy}g(x_2,y_2)\|\|H^{-1}_1 - H^{-1}_2\|\|\nabla^3_{yyy}g(x_1, y_1)\|\notag\\
        &+\|\nabla^2_{xy}g(x_2,y_2)H^{-1}_2\|\|\nabla^3_{yyy}g(x_1, y_1) - \nabla^3_{yyy}g(x_2, y_2)\|\notag\\
        &\le \underbrace{\left(\Lgggg + \Lggg\frac{1}{\mu_g}\Lggg+\Lgg\frac{\Lggg}{\mu_g^2}\Lggg + \Lgg\frac{1}{\mu_g}\Lgggg\right)}_{L_{H^{-1}}}\|(x_1,y_1) - (x_2,y_2)\|. \nonumber
    \end{align}
    Thus, we could bound them separately. 
    
    \begin{align}
        \Delta_1 &\le \underbrace{\Lfff}_{L_{\Phi, 1}}\|(x_1, y_1) - (x_2, y_2)\|\nonumber,\\
        \Delta_2 &\le \|\nabla^2_{xy}g(x_1, y_1) - \nabla^2_{xy}g(x_2, y_2)\|\|H_1^{-1}\nabla^2_{xy}f(x_1, y_1)\| \notag\\
        &\quad+ \|\nabla^2_{xy}g(x_2, y_2)\|\|H_1^{-1} - H_2^{-1}\|\|\nabla^2_{xy}f(x_1, y_1)\| \notag\\
        &\quad+ \|\nabla^2_{xy}g(x_2, y_2)H_2^{-1}\|\|\nabla^2_{xy}f(x_1, y_1) - \nabla^2_{xy}f(x_2, y_2)\|\notag\\
        &\le \underbrace{\left(\Lggg\frac{1}{\mu_g}\Lff + \Lgg\frac{\Lggg}{\mu_g^2}\Lff + \Lgg\frac{1}{\mu_g}\Lfff\right)}_{L_{\Phi,2}}\|(x_1, y_1) - (x_2, y_2)\|\nonumber,\\
        \Delta_3 &\le \|\nabla^2_{xy}g(x_1, y_1) - \nabla^2_{xy}g(x_2, y_2)\|\cdot\|H_1^{-1}\left(\nabla^2_{xy}f(x_1, y_1) - \nabla^2_{xy}g(x_1, y_1)H_1^{-1}\nabla^2_{yy}f(x_1, y_1)\right)\| \notag\\
        &\quad+ \|\nabla^2_{xy}g(x_2, y_2)\|\|H_1^{-1} - H_2^{-1}\|\|\left(\nabla^2_{xy}f(x_1, y_1) - \nabla^2_{xy}g(x_1, yy_1)H_1^{-1}\nabla^2_{yy}f(x_1, y_1)\right)\|\notag\\
        &\quad+ \|\nabla^2_{xy}g(x_2, y_2)H_2^{-1}\|\cdot[\|\nabla^2_{xy}f(x_1, y_1) - \nabla^2_{xy}f(x_2, y_2)\| \notag\\
        &\quad+ \|\nabla^2_{xy}g(x_1, y_1)H_1^{-1}\nabla^2_{yy}f(x_1, y_1) - \nabla^2_{xy}g(x_2, y_2)H_2^{-1}\nabla^2_{yy}f(x_2, y_2)\|] \notag\\ 
        &\le \underbrace{(\dfrac{\Lggg}{\mu_g}(\Lff + \frac{\Lgg}{\mu_g}\Lff) + \Lgg\dfrac{\Lggg}{\mu_g^2}(\Lff + \frac{\Lgg}{\mu_g}\Lff)+ \frac{\Lgg}{\mu_g}(\Lfff + (\frac{\Lggg}{\mu_g}\Lff + \Lgg\frac{\Lggg}{\mu_g^2}\Lff + \frac{\Lgg}{\mu_g}\Lfff)))}_{L_{\Phi,3}} \notag\\
        &\quad\|(x_1, y_1) - (x_2, y_2)\|\nonumber,\\
        \Delta_4 &\le \|\nabla^2_{xy}g(x_1, y_1) - \nabla^2_{xy}g(x_2, y_2)\|\cdot\|\left(H_1^{-1}\nabla_{yxy} g(x_1, y_1)H_1^{-1}\right)\nabla_yf(x_1, y_1)\|\notag\\
        &\quad+ \|\nabla^2_{xy}g(x_2, y_2)\|\cdot\|H_1^{-1} - H_2^{-1}\|\cdot\|\left(\nabla_{yxy} g(x_1, y_1)H_1^{-1}\right)\nabla_yf(x_1, y_1)\|\notag\\
        &\quad+ \|\nabla^2_{xy}g(x_2, y_2)H_2^{-1}\|\cdot\|\nabla_{yxy} g(x_1, y_1) - \nabla_{yxy} g(x_2, y_2)\|\cdot\|\left(H_1^{-1}\right)\nabla_yf(x_1, y_1)\|\notag\\
        &\quad+ \|\nabla^2_{xy}g(x_2, y_2)\left(H_2^{-1}\nabla_{yxy} g(x_2, y_2)\right)\|\cdot\|H_1^{-1} - H_2^{-1}\|\cdot\|\nabla_yf(x_1, y_1)\|\notag\\
        &\quad+ \|\nabla^2_{xy}g(x_2, y_2)\left(H_2^{-1}\nabla_{yxy} g(x_2, y_2)H_2^{-1}\right)\|\cdot\|\nabla_yf(x_1, y_1) - \nabla_yf(x_2, y_2)\|\notag\\
        &\le \underbrace{\left(\Lggg\dfrac{\Lggg}{\mu_g^2}\Lf + \Lgg\dfrac{\Lggg}{\mu_g^2}\dfrac{\Lggg}{\mu_g}\Lf + \dfrac{\Lgg}{\mu_g}\Lgggg\dfrac{\Lf}{\mu_g} + \Lgg\dfrac{\Lggg}{\mu_g}\dfrac{\Lggg}{\mu_g^2}\Lf + \Lgg\dfrac{\Lggg}{\mu_g^2}\Lff\right)}_{L_{\Phi,4}}\notag\\
        &\qquad\cdot\|(x_1, y_1) - (x_2, y_2)\|\nonumber,\\
        \Delta_5 &\le \|\nabla^3_{xxy}g(x_1, y_1) - \nabla^3_{xxy}g(x_2, y_2)\|\cdot\|H_1^{-1}\nabla_yf(x_1, y_1)\|\notag\\
        &\quad+ \|\nabla^2_{xy}g(x_1,y_1)H_1^{-1}\nabla^3_{yxy}g(x_1, y_1) - \nabla^2_{xy}g(x_2,y_2)H_2^{-1}\nabla^3_{yxy}g(x_2, y_2)\|\cdot\|H_1^{-1}\nabla_yf(x_1, y_1)\|\notag\\
        &\quad+ \|\left(\nabla^3_{xxy}g(x_2, y_2) -\nabla^2_{xy}g(x_2,y_2)H_2^{-1}\nabla^3_{yxy}g(x_2, y_2)\right)\|\cdot\|H_1^{-1} - H_2^{-1}\|\cdot\|\nabla_yf(x_1, y_1)\|\notag\\
        &\quad+\|\left(\nabla^3_{xxy}g(x_2, y_2) -\nabla^2_{xy}g(x_2,y_2)H_2^{-1}\nabla^3_{yxy}g(x_2, y_2)\right)H_2^{-1}\|\cdot\|\nabla_yf(x_1, y_1) - \nabla_yf(x_2, y_2)\|\notag\\
        &\le \underbrace{\left((\Lgggg + \Lggg\dfrac{\Lggg}{\mu_g} + \Lgg\dfrac{\Lggg}{\mu_g^2}\Lggg + \dfrac{\Lgg}{\mu_g}\Lgggg)\dfrac{\Lf}{\mu_g} + (\Lggg+\Lgg\dfrac{\Lggg}{\mu_g})\frac{\Lggg}{\mu_g^2}\Lf
        (\Lggg+\Lgg\dfrac{\Lggg}{\mu_g})\dfrac{\Lff}{\mu_g}\right)}_{L_{\Phi,5}}\notag\\&\quad\|(x_1, y_1) - (x_2, y_2)\|\nonumber.
    \end{align}
    By letting $L_{\widehat{\nabla}^2_{xx}\Phi} = L_{\Phi,1} + L_{\Phi,2} + L_{\Phi,3} + L_{\Phi,4} + L_{\Phi,5}$, we have \begin{align}
        \|\widehat{\nabla}^2_{xx}\Phi(x_1, y_1) - \widehat{\nabla}^2_{xx}\Phi(x_2, y_2)\| \le L_{\widehat{\nabla}^2_{xx}\Phi} \|(x_1, y_1) - (x_2, y_2)\|.\nonumber
    \end{align}
\end{proof}

\subsection{Proof of Lemma \ref{lemma:GH_lip}}\label{sec:pf_GHPhi}
\begin{proof}
Denote $H = \nabla^2_{yy}g(x, y^*(x))$ and $\tilde{H} = \nabla^2_{yy}g(x, \tilde{y}(x))$. First, let us estimate the gradient error:  
\begin{align}
    &\quad \|\bm{G} - \nabla_x\Phi(x)\| \notag\\
    &= \|\nabla_x f(x,\tilde{y}(x))-\nabla_{yx}^2g(x,\tilde{y}(x))\tilde{H}^{-1}\nabla_y f(x,\tilde{y}(x)) - \left(\nabla_x f(x,y^*(x))-\nabla_{yx}^2g(x,y^*(x))H^{-1}\nabla_y f(x,y^*(x))\right)\| \notag\\
    &\le \|\nabla_x f(x,\tilde{y}(x)) - \nabla_x f(x,y^*(x))\| + \|\nabla_{yx}^2g(x,\tilde{y}(x))\tilde{H}^{-1}\nabla_y f(x,\tilde{y}(x)) - \nabla_{yx}^2g(x,y^*(x))H^{-1}\nabla_y f(x,y^*(x))\| \notag\\
    &\le \|\nabla_x f(x,\tilde{y}(x)) - \nabla_x f(x,y^*(x))\| + \|\nabla_{yx}^2g(x,\tilde{y}(x))\tilde{H}^{-1}\nabla_y f(x,\tilde{y}(x)) - \nabla_{yx}^2g(x,\tilde{y}(x))\tilde{H}^{-1}\nabla_y f(x,y^*(x))\| \notag\\
    &+\|\nabla_{yx}^2g(x,\tilde{y}(x))\tilde{H}^{-1}\nabla_y f(x,y^*(x)) - \nabla_{yx}^2g(x,\tilde{y}(x))H^{-1}\nabla_y f(x,y^*(x))\| \notag\\
    &+ \|\nabla_{yx}^2g(x,\tilde{y}(x))H^{-1}\nabla_y f(x,y^*(x)) - \nabla_{yx}^2g(x,y^*(x))H^{-1}\nabla_y f(x,y^*(x))\|\notag\\
    &\le  \underbrace{\left(\Lf + \Lf\dfrac{\Lg}{\mu_g} + \Lf\dfrac{2\Lgg}{\mu_g^2}\Lg + \Lf\dfrac{1}{\mu_g}\Lgg\right)}_{L_G}\|\tilde{y}(x) - y^*(x)\|,\nonumber
\end{align}
where the last inequality comes from Assumption \ref{assumption:general2} and $\|\tilde{H}^{-1} - H^{-1}\| = \|\tilde H^{-1}(H - \tilde H)H^{-1}\| \le \dfrac{2\Lgg}{\mu_g^2}$, 
Next, for the Hessian error estimation, we directly derive from Lemma \ref{lem:lip_xy} that
\begin{align}
    \|\bm{H} - \nabla^2_{xx}\Phi(x)\| = \|\widehat{\nabla}^2_{xx}\Phi(x, \tilde{y}(x)) - \widehat{\nabla}^2_{xx}\Phi(x, y^*(x))\| \le L_H(:= L_{\widehat{\nabla}^2_{xx}\Phi})\|\tilde{y}(x) - y^*(x)\|.\nonumber
\end{align}
\end{proof}

Now, with Lemma \ref{lem:lip_xy} and Lemma \ref{lemma:GH_lip} in place, we are ready to prove 
Proposition \ref{prop:lPhi}.

\subsection{Proof of Proposition \ref{prop:lPhi}}\label{sec:pf_prop_lPhi}
\begin{proof}
First, $y^*(x)$ is Lipschitz continuous with a constant $L_{y,1}:= \dfrac{\Lgg}{\mu_g}$ because
\begin{align}
    \|\nabla_x y^*(x)\| &= \|\nabla^2_{xy}g(x,y^*(x))\|\cdot\|(\nabla^2_{yy}g(x,y^*(x)))^{-1}\| \le L_{y, 1} . \nonumber
    \end{align}
    Secondly, by Lemma \ref{lem:lip_xy}, and denoting $(x_1, y_1) = (x_1, y^*(x_1))$ and $(x_2, y_2) = (x_2, y^*(x_2))$, we have 
    \begin{align}
        \|\nabla_{xx}^2 \Phi(x_1) - \nabla_{xx}^2 \Phi(x_2)\| &= \|\widehat{\nabla}^2_{xx}\Phi(x_1, y^*(x_1)) - \widehat{\nabla}^2_{xx}\Phi(x_2, y^*(x_2))\| \notag\\
        &\le L_{\widehat{\nabla}^2_{xx}\Phi} \|(x_1, y^*(x_1)) - (x_2, y^*(x_2))\|\notag\\
        &\le L_{\nabla^2\Phi} (:= L_{\widehat{\nabla}^2_{xx}\Phi}(1+L_{y, 1})) \|x_1 - x_2\|. \nonumber
    \end{align}
\end{proof}

\section{Proof of Results in Section \ref{sec:3}} \label{sec:pf_of_dl}

\begin{proof}
Let $\delta_{k} = \min\left\{\frac{a}{L_{\nabla\Phi}}\|s_k\|^2, \frac{b}{L_{\nabla^2\Phi}}\|s_k\|\right\}$ to denote the inner stopping threshold at iteration $k$. 
Suppose that for all iterations $k = 1, \dots, K$, the step size satisfies $\|s_k\| > \frac{\sqrt{\varepsilon}}{c}$, where $c = \max\left\{ \sqrt{\frac{M}{2}+a+b+\frac{L_{\nabla^2\Phi}}{2}}, \ \frac{M}{2}+b+L_{\nabla^2\Phi}\right\}$. If this condition does not hold for some $k$ (i.e., $\|s_k\| \le \frac{\sqrt{\varepsilon}}{c}$), then $x_k$ is already an $\varepsilon$-SOSP. To see this, observe
\begin{align}
    \|\nabla \Phi(x_{k})\| &\le \|\nabla\Phi(x_{k-1}) + \nabla^2\Phi(x_{k-1})s_k\| + \frac{L_{\nabla^2\Phi}}2\|s_k\|^2 \nonumber\\
    &\le\|\bm G_{k-1} + \bm H_{k-1}s_k\| + \|\nabla\Phi(x_{k-1}) - \bm G_{k-1}\| + \|\nabla^2\Phi(x_{k-1}) - \bm H_{k-1}\|\|s_k\| + \frac{L_{\nabla^2\Phi}}2\|s_k\|^2 \nonumber\\
    &\le \left(\frac{M}{2} + a + b + \frac{L_{\nabla^2\Phi}}2\right)\|s_k\|^2 \le \varepsilon, \nonumber\\
    \nabla^2\Phi(x_k) &\succcurlyeq \bm H_{k-1} - \|\nabla^2\Phi(x_{k}) - \nabla^2\Phi(x_{k-1})\|\bm{I} - \|\nabla^2\Phi(x_{k-1}) - \bm H_{k-1}\|\bm{I} \nonumber\\
    &\succcurlyeq -\left(\dfrac{M}{2} + b + L_{\nabla^2\Phi}\right)\|s_k\|\bm I
\end{align}
Thus, if the algorithm has not been terminated, then it means that we must have $\|s_k\| > \frac{\sqrt{\varepsilon}}{c}$ at that point. This implies a lower bound on the threshold
\begin{align}
    \delta_k \ge \min\left\{\frac{a\varepsilon}{c^2L_{\nabla\Phi}}, \frac{b\sqrt{\varepsilon}}{cL_{\nabla^2\Phi}}\right\} := \delta_\varepsilon.\nonumber
\end{align}
We now proceed to bound the total inner iterations $\sum_{k=1}^K N_k$. From \cite{wang2018note}, we have the summation bound $\sum_{k=1}^K\|s_k\|^3 \le \left(\frac{3M - 2L_{\nabla^2\Phi}}{12} - 2a-2b\right)^{-1}\left(\Phi(x_0) - \Phi^* + (a+b)\|s_0\|^3\right)$, which implies $\frac{1}{K}\sum_{k=1}^K\|s_k\| \le \left(\frac{1}{K}\sum_{k=1}^K \|s_k\|^3\right)^{1/3} = \mathcal O(K^{-1/3}) $.

Since $g(x,\cdot)$ is $\mu_g$-strongly convex and $\Lgg$-smooth, the inner gradient descent step with stepsize $\beta = 1/\Lgg$ satisfies $\|y_{k,t}-y^*(x_k)\| \le \left(1-\frac{\mu_g}{\Lgg}\right)^t \|y_{k,0}-y^*(x_k)\|$ for all $t\ge 0$. Using the triangle inequality and the initialization $y_{k,0} = y_{k-1, N_{k-1}}$, we have
\begin{align}
    \|y_{k,0}-y^*(x_k)\| &= \|y_{k-1,N_{k-1}}-y^*(x_k)\| \le \|y_{k-1,N_{k-1}}-y^*(x_{k-1})\| + \|y^*(x_{k-1})-y^*(x_k)\| \nonumber\\
    &\le \delta_{k-1} + L_{y,1}\|s_k\|.\nonumber
\end{align}
By the residual stopping criterion and the $\mu_g$-strong convexity of $g(x_k,\cdot)$, we have $\|y_{k,N_k}-y^*(x_k)\|\le\dfrac{1}{\mu_g}\|\nabla_y g(x_k,y_{k,N_k})\|\le\delta_k$. Moreover, since $g(x_k,\cdot)$ is $\Lgg$-smooth, the residual stopping criterion is guaranteed if $\|y_{k,N_k}-y^*(x_k)\|\le\frac{\mu_g}{\Lgg}\delta_k$.
Therefore, it is sufficient that $\left(1-\frac{\mu_g}{\Lgg}\right)^{N_k}
    \left(\delta_{k-1}+L_{y,1}\|s_k\|\right)
    \le
    \frac{\mu_g}{\Lgg}\delta_k.$
Solving for $N_k$, we obtain
\begin{align}
    N_k
    \le
    \frac{1}{-\log(1-\frac{\mu_g}{\Lgg})}
    \log\left(
    \frac{\Lgg(\delta_{k-1}+L_{y,1}\|s_k\|)}
    {\mu_g\delta_k}
    \right)\le
    \frac{\Lgg}{\mu_g}
    \left(
    \log\left(\frac{\delta_{k-1}}{\delta_k}\right)
    +
    \log\left(1+\frac{L_{y,1}\|s_k\|}{\delta_{k-1}}\right)
    +
    \log\left(\frac{\Lgg}{\mu_g}\right)
    \right).
    \nonumber
\end{align}
Summing over $k=1$ to $K$, we have
\begin{align}
    \sum_{k=1}^K N_k &\le \frac{\Lgg}{\mu_g}\log\left(\frac{\delta_0}{\delta_K}\right) + \sum_{k=1}^K \frac{\Lgg}{\mu_g}\log\left(1 + L_{y,1}\frac{\|s_k\|}{\delta_\varepsilon}\right)+
    \frac{K \Lgg}{\mu_g}
    \log\left(\frac{\Lgg}{\mu_g}\right) \nonumber\\
    &\le = \frac{\Lgg}{\mu_g}\log\left(\frac{\delta_0}{\delta_K}\right) + K\log\left(1 + \dfrac{L_{y,1}}{\delta_\varepsilon}\cdot \dfrac{1}{K}\sum_{k=1}^K\|s_k\|\right)+
    \frac{K \Lgg}{\mu_g}
    \log\left(\frac{\Lgg}{\mu_g}\right) \nonumber\\
    &=\mathcal{O}(\varepsilon^{-1.5}\log(\varepsilon^{-1})),\nonumber
\end{align}
since $\delta_\varepsilon = \mathcal{O}(\varepsilon)$ and $K = \mathcal O(\varepsilon^{-3/2})$, we have $\|s_k\| = \mathcal O(K^{-1/3}) = \mathcal{O}(\varepsilon^{1/2})$. This completes the proof.
\end{proof}

\section{Proof of the Results in Section \ref{sec:4}} \label{sec:pf_of_sl}

\subsection{Technical Lemmas}
First we introduce two technical lemmas below, showing the stepwise contraction of gradient descent and Newton on the strongly convex lower level problem.

\begin{proposition}[\cite{nesterov2013introductory}, Theorem 2.1.15]\label{lem:gd}
    Under Assumptions \ref{assumption:general1}, \ref{assumption:nonlinear2} and \ref{assumption:general2}, for the lower level gradient descent step $y_{k+1} = y_k - \beta\nabla_yg(x_k, y_k)$ and $\beta\le \frac{2}{\mu_g+\Lgg}$, denoting $\rho_g = \frac{2\mu_g\Lgg}{\mu_g + \Lgg}$ and $S_H := \sqrt{1 - \rho_g\beta}\in(0, 1)$, we have
    \begin{align}
        \|y_{k+1} - y^*(x_k)\| &\le S_H \|y_{k} - y^*(x_k)\|.\nonumber
    \end{align}
\end{proposition}

\begin{lemma}\label{lem:newton}
Under Assumptions \ref{assumption:general1}, \ref{assumption:nonlinear2} and \ref{assumption:general2}, for the lower level newton step $\hat{y}_{k+1} = y_{k+1} - \left(\nabla_{yy}^2 g(x_k, y_{k+1})\right)^{-1}\nabla_y g(x_k, y_{k+1})$, denoting $S_G = \frac{\Lggg}{2\mu_g}(1-\rho_g\beta)$ we have:
\begin{align}
    \|\hat{y}_{k+1} - y^*(x_k)\| \le S_G\|y_{k} - y^*(x_k)\|^2. \nonumber
\end{align}
\end{lemma}

\begin{proof}
The optimality condition of $y^*(x_k)$ ensures that $\nabla_yg(x_k, y^*(x_k)) = 0$.
For all fixed $y\in\RR^n$, use Taylor expansion of $\nabla_y g(x_k, y)$ around $y$ and there exists $\xi$ such that \begin{align*}
    \nabla_y g(x_k, y^*(x_k)) = \nabla_yg(x_k, y) + \nabla^2_{yy} g(x_k, y)(y-y^*(x_k)) + (\xi-y^*(x_k))^\top\nabla^3_{yyy}g(x_k, (\xi-y^*(x_k)) = 0
\end{align*}
Then, setting $y = \hat y_{k+1}$, and after the Newton step, we have
\begin{align}
    \|\hat{y}_{k+1} - y^*(x_k)\| &= \|y_{k+1} - y^*(x_k) - \nabla^2_{yy}g(x_k, y_{k+1})^{-1}(y_{k+1}-y^*(x_k))\| \nonumber\\
    &= \|y_{k+1} - y^*(x_k) - (y_{k+1} - y^*(x_k)) \nonumber\\
    &\qquad- \frac{1}{2}\left(\nabla_{yy}^2 g(x_k, y_{k+1})\right)^{-1}(y^*(x_k) - y_{k+1})^T\nabla_{yyy}^3 g(x_k, \zeta_{k+1}) (y^*(x_k) - y_{k+1})\| \nonumber\\
    &\le \frac{1}{2}\|\left(\nabla_{yy}^2 g(x_k, y_{k+1})\right)^{-1}\|\Lggg\|y_{k+1} - y^*(x_k)\|^2 \le \frac{\Lggg}{2\mu_g}\|y_{k+1} - y^*(x_k)\|^2. \nonumber
    \end{align}
Letting $S_G = \dfrac{\Lggg}{2\mu_g}(1-\rho_g\beta_k)$, the desired result follows.
\end{proof}

\subsection{Proof of Theorem \ref{thm:main}}\label{sec:pf_of_main}
\begin{proof}
The one-step decrease of CRN for the hyperfunction can be estimated as:
\begin{align}
    &\quad\ \Phi(x_{k+1}) - \Phi(x_k) \nonumber\\&\le \widehat{\nabla}_x\Phi(x_k, \hat{y}_{k+1})^\top s_{k+1} + \frac{1}{2} s_{k+1}^\top\widehat{\nabla}_{xx}^2\Phi(x_k,y_{k+1})s_{k+1} + \frac{M}{6}\|s_{k+1}\|^3\nonumber\\
    &\quad+ \left(\widehat{\nabla}_x\Phi(x_k, y^*(x_k)) - \widehat{\nabla}_x\Phi(x_k, \hat{y}_{k+1})\right)^\top s_{k+1} \nonumber\\
    &\quad+ \frac{1}{2}s_{k+1}^\top\left(\widehat{\nabla}_{xx}^2\Phi(x_k, y^*(x_k)) - \widehat{\nabla}_{xx}^2\Phi(x_k,y_{k+1})\right)s_{k+1} + \frac{L_{\nabla^2\Phi}-M}{6}\|s_{k+1}\|^3 \nonumber\\
    &\le -\frac{3M-2L_{\nabla^2\Phi}}{12}\|s_{k+1}\|^3 + L_G\|y^*(x_k) - \hat{y}_{k+1}\|\|s_{k+1}\| + \frac{1}{2}L_H\|y^*(x_k) - y_{k+1}\|\|s_{k+1}\|^2 \label{eq:main_1}\\
    &\le -\frac{3M-2L_{\nabla^2\Phi}}{12}\|s_{k+1}\|^3 + L_GS_G\|y^*(x_k) - y_{k}\|^2\|s_{k+1}\|  + \frac{1}{2}L_HS_H\|y^*(x_k) - y_{k}\|\|s_{k+1}\|^2 \label{eq:main_2}\\
    &\le -\frac{3M-2L_{\nabla^2\Phi}}{12}\|s_{k+1}\|^3 + \frac{2L_GS_G + \frac{1}{2}L_HS_H}{3}\|y^*(x_k) - y_{k}\|^3  + \frac{L_GS_G + L_HS_H}{3}\|s_{k+1}\|^3,\label{eq:main_pf_1}
\end{align}
where \eqref{eq:main_1} applies Lemma \ref{lemma:GH_lip} and the CRN descent lemma, \eqref{eq:main_2} uses $S_G$ and $S_H$, and \eqref{eq:main_pf_1} applies the AM-GM inequality. Using the generalized AM-GM inequality $\|a+b\|^3 \le (1+2\gamma + \gamma^2)\|a\|^3 + \frac{2+\gamma+\gamma^2}{\gamma^2} \|b\|^3$, we obtain the following error recursion:
\begin{align}
   &\quad\ \|y_{k+1} - y^*(x_{k+1})\|^3 - \|y_{k} - y^*(x_{k})\|^3 \nonumber\\&\le \left(\|y_{k+1} - y^*(x_{k})\| + \|y^*(x_{k+1}) - y^*(x_{k})\|\right)^3 - \|y_{k} - y^*(x_{k})\|^3\nonumber\\
   &\le (1+2\gamma + \gamma^2)\|y_{k+1} - y^*(x_{k})\|^3 + \frac{2+\gamma+\gamma^2}{\gamma^2}\|y^*(x_{k+1}) - y^*(x_{k})\|^3 - \|y_{k} - y^*(x_{k})\|^3\nonumber\\
   &\le \left((1+2\gamma + \gamma^2)S_H^3-1\right)\|y_{k} - y^*(x_{k})\|^3 + \frac{2+\gamma+\gamma^2}{\gamma^2}L^3_{y,1}\|s_{k+1}\|^3.\label{eq:main_pf_2}
\end{align}
Summing up \eqref{eq:main_pf_1} and \eqref{eq:main_pf_2}, we are in a position to introduce the Lyapunov function $\bbV_k := \Phi(x_k) + \tau \|y_{k} - y^*(x_{k})\|^3$, which satisfies the following descent property:
\begin{align}
    \bbV_{k+1} - \bbV_k &\le -a \|y_{k} - y^*(x_{k})\|^3 - b \|s_{k+1}\|^3,\label{eq:descent_1}
\end{align}
where 
\(
a := \tau\left(1-(1+2\gamma + \gamma^2)S_H^3\right) - \frac{2L_GS_G + \frac{1}{2}L_HS_H}{3} \)
and 
\(b := \frac{3M-2L_{\nabla^2\Phi}}{12} - \frac{L_GS_G + L_HS_H}{3} - \tau\frac{2+\gamma+\gamma^2}{\gamma^2} L^3_{y,1}.
\)
To ensure that \eqref{eq:descent_1} provides a proper descent property, we choose constant $\gamma$ to satisfy $1-(1+2\gamma+\gamma^2)S_H^3 > 0$.  
Furthermore, we choose $\tau$ and $M$ to ensure $\alpha, \beta > 0$, and the following inequalities hold as well: 
\begin{align}
    &\tau(1-(1+2\gamma+\gamma^2)S_H^3) > \frac{2L_GS_G + \frac{1}{2}L_HS_H}{3} \label{eq:satif_tau} \\
    &M > \frac{2L_{\nabla^2\Phi}}{3}+4\left(\frac{L_GS_G + L_HS_H}{3} + \tau \frac{2+\gamma+\gamma^2}{\gamma^2} L^3_{y,1}\right). \label{eq:satif_M}
\end{align}
Summing \eqref{eq:descent_1} over $k$ and denoting $\bbV := \inf_{x, y}\left\{\Phi(x) + \tau\|y - y^*(x)\|^3\right\}$ 
yields 
\begin{align*}
    K \min_{k\in[K]} \left[ a\|y_{k} - y^*(x_{k})\|^3 + b\|s_{k+1}\|^3 \right] \le \sum_{k=1}^K \left(\bbV(x_k) - \bbV(x_{k+1})\right) \le \bbV(x_1) - \bbV^* =: \Delta.
\end{align*}
Thus, there exists $\bar k\in[K]$ such that
\begin{align}\label{eq:final_rate}
     \|y_{\bar k} - y^*(x_{\bar k})\| \le \sqrt[3]{\frac{\Delta}{a K}} = \mathcal{O}(K^{-1/3}), \qquad \|s_{\bar k+1}\| \le \sqrt[3]{\frac{\Delta}{b K}} = \mathcal{O}(K^{-1/3}),
\end{align}
Notice that we are unable to access $y^*(x)$, so at the algorithm output we need a more reasonable criterion $\nabla_yg(x_k, y_k)$ instead of $\|y_k - y^*(x_k)\|$. Since $\nabla_y g(x, \cdot)$ is $\mu_g$-strongly convex and $\Lgg$-smooth in $y$, we have, for every $k$, \begin{align}
    \mu_g\|y_k - y^*(x_k)\| \le \|\nabla_y g(x_k, y_k)\| \le \Lgg\|y_k - y^*(x_k)\|. \nonumber
\end{align}
Therefore let $k = \bar k$, we have \begin{align}
    \|\nabla_yg(x_{\bar k}, y_{\bar k})\|^3 + \|s_{\bar k+1}\|^3 &\le \max\left\{\dfrac{\Lgg^3}{a}, \dfrac{1}{b}\right\}\left[a\|y_{\bar k} - y^*(x_{\bar k})\|^3 + b\|s_{\bar k+1}\|^3\right] \nonumber,\\
    \|\nabla_yg(x_{\bar k}, y_{\bar k})\|^3 + \|s_{\bar k+1}\|^3 &\ge \min\left\{\dfrac{\mu_g^3}{a}, \dfrac{1}{b}\right\}\left[a\|y_{\bar k} - y^*(x_{\bar k})\|^3 + b\|s_{\bar k+1}\|^3\right],\nonumber
\end{align}
and such $\bar k$ can be identified as $\bar k\in\argmin_{k}\left\{\|\nabla_yg(x_{k}, y_{k})\|^3 + \|s_{k+1}\|^3\right\}$.
We can now show the convergence rate:
\begin{align}
    \|\nabla_x \Phi(x_{\bar k+1})\| &= \left\|\nabla_x \Phi(x_{\bar k+1}) - \left(\bm{G}_{\bar k} + \bm{H}_{\bar k}s_{\bar k+1} + \frac{M}{2}\|s_{\bar k+1}\|s_{\bar k+1}\right)\right\| \nonumber\\
    &\le \left\|\nabla_x \Phi(x_{\bar k+1}) - \nabla_x \Phi(x_{\bar k}) - \nabla_{xx}^2\Phi(x_{\bar k})s_{\bar k+1}\right\| + \|\bm{G}_{\bar k} - \nabla \Phi_x(x_{\bar k})\| \nonumber\\
    &\quad + \|(\bm{H}_{\bar k} - \nabla_{xx}^2\Phi(x_{\bar k}))s_{\bar k+1}\| + \frac{M}{2}\|s_{\bar k+1}\|^2\nonumber\\
    &\le \dfrac{L_{\nabla^2\Phi}}{2}\|s_{\bar k+1}\|^2 + L_GS_G\|y^*(x_{\bar k}) - y_{\bar k}\|^2 + L_HS_H\|y^*(x_{\bar k}) - y_{\bar k}\|\|s_{\bar k+1}\| + \frac{M}{2}\|s_{\bar k+1}\|^2 \nonumber\\ &= \mathcal{O}(K^{-2/3}). \nonumber
\end{align}
Additionally, we have
\begin{align*}
    \nabla^2\Phi(x_{\bar k+1}) &\succcurlyeq \bm{H}_{\bar k} - \|\nabla^2\Phi(x_{\bar k+1}) - \nabla^2\Phi(x_{\bar k})\|\bm{I} - \|\nabla^2\Phi(x_{\bar k}) - {H}_{\bar k}\|\bm{I} \\
    &\succcurlyeq -\left(\left(\frac{M}{2} + L_H\right)\|s_{\bar k+1}\| + L_HS_H\|y_{\bar k} - y^*(x_{\bar k})\|\right)\bm{I}.
\end{align*} 
Finally, from \eqref{eq:final_rate} it follows that 
\begin{align}
    \lambda_{\min}\left(\nabla^2\Phi(x_{\bar k+1})\right) \ge -\left(\frac{M}{2} + L_H\right)\|s_{\bar k+1}\| - L_HS_H\|y_{\bar k} - y^*(x_{\bar k})\| = \mathcal{O}(K^{-1/3}). \nonumber
\end{align}
Therefore, the oracle complexity of locating an $\varepsilon$-SOSP is $\mathcal{O}(\varepsilon^{-1.5})$.
\end{proof}

\end{document}